\documentclass[a4paper,12pt]{article}

\usepackage[left=2cm,right=2cm, top=2cm,bottom=2cm,bindingoffset=0cm]{geometry}

\usepackage{verbatim}
\usepackage{amsmath}
\usepackage{amsthm}
\usepackage{amssymb}
\usepackage{delarray}
\usepackage{cite}
\usepackage{hyperref}
\usepackage{mathrsfs}

\newcommand{\al}{\alpha}

\newcommand{\ga}{\gamma}

\newcommand{\la}{\lambda}

\newcommand{\eps}{\varepsilon}

\newcommand{\iy}{\infty}

\usepackage{graphicx}
\usepackage{tikz}
\usetikzlibrary{patterns}
\usepackage{caption}
\DeclareCaptionLabelSeparator{dot}{. }
\theoremstyle{plain}

\newtheorem{thm}{Theorem}
\newtheorem{lem}{Lemma}

\theoremstyle{definition}

\newtheorem{example}{Example}
\newtheorem{alg}{Algorithm}
\newtheorem{prob}{Problem}
\newtheorem{remark}{Remark}

 \allowdisplaybreaks

\begin{document}

\begin{center}
{\large\bf An inverse problem for Sturm-Liouville operators on trees with partial information given on the potentials
}
\\[0.2cm]
{\bf Natalia P. Bondarenko} \\[0.2cm]
\end{center}

\vspace{0.5cm}

{\bf Abstract.} We consider Sturm-Liouville operators on geometrical graphs without cycles (trees) with singular potentials from the class $W_2^{-1}$. We suppose that the potentials are known on a part of the graph, and study the so-called partial inverse problem, which consists in recovering the potentials on the remaining part of the graph from some parts of several spectra. The main results of the paper are the uniqueness theorem and a constructive procedure for the solution of the partial inverse problem. Our method is based on the completeness and the Riesz-basis property of special systems of vector functions, and the reduction of the partial inverse problem to the complete one on a part of the graph.
 
\medskip


\medskip


\vspace{1cm}

{\large \bf 1. Introduction}

\bigskip

The paper concerns the inverse problems theory for quantum graphs, i.e. metric graphs, equipped by differential operators.
Quantum graphs arise in applications as models of wave propagation through a domain being a thin neighborhood of a graph. Such models are studied in
organic chemistry, mesoscopic physics and nanotechnology, the theories of photonic crystals, waveguides, quantum chaos,
and other branches of science and engineering (see, for example, \cite{RS53, Griff53, Dat99, Imry97, Mur01, KK02, ES89, CLMT97, KS97}).
Moreover, the quantum graphs theory is the source of challenging problems, interesting from the mathematical point of view.

There is an extensive literature, devoted to differential operators on graphs. 
Basics of the spectral theory for quantum graphs and further references to applications can be found 
in the monographs \cite{BCFK06, BK13, GS06, Post12, PPP04} and the brilliant collection of symposia proceedings \cite{Exner08}.
Shorter elementary introductions to quantum graphs are provided in \cite{Kuch04, Ber16}.
The survey \cite{Yur16} contains a good overview of the inverse spectral problems theory for differential operators on graphs.
Such problems consist in recovering quantum graphs from various types of spectral characteristics, and generalize the well-studied 
inverse spectral problems for Sturm-Liouville, or Schr\"odinger operators on intervals \cite{Mar77, Lev84, PT87, FY01}.
However, \cite{Yur16} is mostly focused on the reconstruction of differential operator coefficients (for example, potentials of Sturm-Liouville operators), 
while there are also results on recovering a graph structure and matching conditions from some spectral data 
(see \cite{GS01, KS00, KN05, AKN10, EK12}).

In papers \cite{Yur05, Yur06}, Yurko has studied inverse problems for Sturm-Liouville operators on graphs without cycles (called trees) by spectral data of three different types:
eigenvalues together with weight numbers, Weyl functions and several spectra. Uniqueness theorems have been proved, 
and constructive solutions for the mentioned inverse problems have been provided, based on the method of spectral mappings \cite{FY01, Yur02}.
Later on, this approach was generalized for arbitrary compact graphs \cite{Yur10}. 

In \cite{Yur05, Yur06} it has been proved, in particular, that 
the Sturm-Liouville potential on a tree is uniquely specified by $b$ spectra, where $b$ is the number of boundary vertices. 
This proposition generalizes the well-known result by Borg \cite{Borg46}, that the Sturm-Liouville operator on a finite interval can be uniquely recovered 
from two spectra, corresponding to boundary conditions, different at one end of the interval. Nevertheless, it is still not clear,
whether the system of $b$ spectra is minimal data, determining the Sturm-Liouville operator on a tree. This question is opened even for
the three-edged star-shaped graph. Moreover, since the system of spectra contains a quite large number of eigenvalue sequences, the question arises,
if it is possible anyway to reduce the data.

One of the possibilities for the Sturm-Liouville operator on an interval is given by Hochstadt-Lieberman theorem \cite{HL78}. 
It states that, if the potential is known a priori on the half of the interval, a spectrum determines it uniquely on the other half. 
Important developments of the results by Hochstadt and Lieberman were obtained in \cite{GS00, Sakh01, MP10, HM04-half}.
Similar ideas can be applied for quantum graphs. However, this will not be a trivial generalization because of complicated graph structure.

Pivovarchik \cite{Piv00} studied inverse problems for the Sturm-Liouville operator on a star-shaped graph
and noticed, that if the potentials are known a priori on all the edges of the graph, except one, then only one spectrum is required for the unique
specification of the unknown potential. Yang \cite{Yang10} has proved, that, in fact, some fractional part of the spectrum
is sufficient for that. Some special cases of recovering the potential on a part of an edge, while it is known on the other part of the graph, 
were studied in \cite{Yur09, Yang11, Yang17, YW17}. 
A constructive method of solution for such partial inverse problems
on star-shaped graphs with equal length edges have been developed in \cite{Bond17-1, Bond17-2, Bond17-preprint, Bond17-3}. 
This method is based on the Riesz-basis property of some systems of vector functions, 
and allows one to prove the local solvability and the stability for the solution of partial inverse problems on graphs \cite{Bond17-1}. 

The goal of this paper is to develop the general theory of partial inverse problems for differential operators on trees of arbitrary structure.
We consider a tree divided into two connected subtrees: ``known'' and ``unknown'' parts. We suppose that the Sturm-Liouville potentials are known
a priori on the first part, and select fractional parts of a spectrum or several spectra, uniquely specifying the potentials on the second part.
Constructing special systems of vector functions, complete in some Banach space, we reduce the partial inverse problem to the complete inverse problem 
on the ``unknown'' part of the tree, which has been studied in \cite{Yur05, Yur06}. This reduction allows us to prove the uniqueness theorem and to provide
an algorithm for reconstruction of the operator on the tree, given the partial information on the potentials.
Note that partial inverse problems on trees were also studied in \cite{BS17}, but there the particularly different situation was discussed,
when the potential was known on the only one edge and unknown on the others, so \cite{BS17} is not closely related to the present paper.

Another feature of this paper is that we consider Sturm-Liouville operators with singular potentials from the Sobolev space $W_2^{-1}(0, \pi)$.
Spectral properties of such operators on intervals have been studied by Savchuk and Shkalikov \cite{SS99, Sav01},
inverse problems theory has been developed by Hryniv and Mykytyuk \cite{HM03, HM04-2spectra, HM04-half, HM04-transform}.
There are only a few works for differential operators with singular coefficients on graphs \cite{FIY08, Bond17-preprint}.
However, for our purposes the class $W_2^{-1}(0, \pi)$ does not cause any additional difficulties, comparing with $L_2(0, \pi)$,
and even makes the presentation of the results clearer, so we work with this more general class.
 
The paper is organized as follows. In Section~2, we introduce the Sturm-Liouville boundary value problem with singular potentials on a tree, investigate properties of its characteristic function and asymptotic behavior of eigenvalues. Section~3 is devoted to the partial inverse problem on the tree (Problem~\ref{prob:pip}) and the general scheme of its solution. At the beginning of Section~3, we study the auxiliary problem (Problem~\ref{prob:ap}) of recovering characteristic functions, associated with the ``unknown part'' of the tree, from some part of the spectrum for the whole tree. The solution of the auxiliary problem reduces our partial inverse problem to the complete inverse problem on the ``unknown part'' of the tree by the system of spectra (Problem~\ref{prob:ip}). The last problem was solved by Yurko and co-authors \cite{Yur05, Yur06, FIY08} by the method of spectral mappings. Applying their results, we arrive at the uniqueness theorem and the constructive solution of the partial inverse problem. The results of Section~3 are quite abstract. The uniqueness theorem (Theorem~\ref{thm:uniq}) requires the completeness of some system of vector functions, and the constructive solution requires its unconditional basicity. In Section~4, we provide sufficient conditions for the completeness of this system, choosing subsequences of eigenvalues with appropriate asymptotics. Section~5 contains two examples of particular quantum graphs, illustrating the general scheme of solution of the partial inverse problem, and simplifications available in certain special cases. 

\bigskip

{\large \bf 2. Asymptotic behavior of eigenvalues and other preliminaries}

\bigskip

In this section, we state the boundary value problem for the Sturm-Liouville equations with singular potentials on a tree and study its spectral properties.
The main results of this section are the representation of the characteristic function, given by Lemma~\ref{lem:Delta}, and the asymptotic formulas for the eigenvalues (Lemma~\ref{lem:asymptla}),
which will be used for analysis of the inverse problem.

Consider a compact tree $G$ with the set of vertices $V$ and the edges $E = \{ e_j \}_{j = 1}^m$.
For each vertex $v \in V$, we denote the set of edges associated with $v$ by $E_v$ and call the size of this set $|E_v|$ the {\it degree} of $v$.
The vertices of degree $1$ are called {\it boundary vertices}, the others are called {\it internal vertices}.  
Denote the set of the boundary vertices of the graph $G$ by $\partial G$ and $\mbox{int}\, G := V \backslash \partial G$.

Let us assume that the edges of $G$ have equal lengths $\pi$. If the lengths of the edges are different, but rationally dependent,
one can achieve the case of equal lengths, dividing the edges by additional vertices. For each edge $e_j$, we introduce a parameter $x_j \in [0, \pi]$.
The value $x_j = 0$ corresponds to one of the vetrices incident to $e_j$, and $x_j = \pi$ corresponds to the other one.

A function on the tree $G$ is a vector function $y = [y_j]_{j = 1}^m$,
where $y_j = y_j(x_j)$, $x_j \in [0, \pi]$, $j = \overline{1, m }$.
Let $q_j$, $j = \overline{1, m}$, be real-valued functions from the class $W_2^{-1}(0, \pi)$, i.e.
$q_j = \sigma_j'$, $\sigma_j \in L_2(0, \pi)$, where the derivative is considered in the sense of distributions. 
The Sturm-Liouville expression
$$
\ell_j y_j := -y_j'' + q_j(x_j) y_j  
$$
on the edge $e_j$ can be understood in the following sense (see \cite{SS99, Sav01, HM03}):
$$
\ell_j y_j = -(y_j^{[1]})' - \sigma_j(x_j) y_j^{[1]} - \sigma_j^2(x_j) y_j,
$$ 
where $y_j^{[1]} = y_j' - \sigma_j y_j$ is a {\it quasi-derivative}, and $y_j$ belongs to the domain
$$
\mathcal D(\ell_j) = \{ y_j \in W_2^1[0, \pi] \colon y_j^{[1]} \in W_1^1[0, \pi], \: \ell_j y_j \in L_2(0, \pi) \}. 
$$

Suppose that the vertex $u \in V$ corresponds to the end $x_j = 0$ of the edge $e_j$ and the vertex $v \in V$ corresponds to 
$x_j = \pi$. Let $[\ga_j]_{j = 1}^m$ be some real constants. Further we use the notations
$$
\begin{array}{ll}
    y_j(u) = y_j(0), \quad & y_j(v) = y_j(\pi),\\
    y^{[1]}_j(u) = -y^{[1]}_j(0), \quad & y^{[1]}_j(v) = y^{[1]}_j(\pi) + \ga_j y_j(\pi).
\end{array}
$$
For $u \in \partial G$ we can omit the index $j$ and write $y(u)$, $y^{[1]}(u)$.

Consider the boundary value problem $L$ on the tree $G$ for the system of Strum-Liouville equations
\begin{equation} \label{eqv}
    (\ell_j y_j)(x_j) = \la y_j (x_j), \quad x_j \in (0, \pi), \: y_j \in \mathcal D(\ell_j), \: j = \overline{1, m},
\end{equation}
with the standard matching conditions 
\begin{equation} \label{MC}
\arraycolsep=1.4pt\def\arraystretch{1.5}
\left.
\begin{array}{ll}    
    y_j(v) = y_k(v), \quad e_j, e_k \in E_v \quad \text{(continuity condition)}, \\
    \sum\limits_{e_j \in E_v} y^{[1]}_j(v) = 0, \quad \text{(Kirchhoff's condition)}
\end{array} \qquad
\right\}
\end{equation}
in the internal vertices $v \in \mbox{int}\, G$, and the Dirichlet $y(v) = 0$ 
or the Neumann $y^{[1]}(v) = 0$ boundary conditions
in the vertices $v \in \partial G$. There can be different types of 
boundary conditions in different vertices. 
Fix some boundary conditions of the described type, and denote them by BC.
If we further consider any splitting of $G$ into subtrees,
we suppose that the parametrization by $x_j$ is inherited, 
and the matching and the boundary conditions for the subtrees remains the same
in the vertices, those are not influenced by splitting. 

Let $C_j(x_j, \la)$ and $S_j(x_j, \la)$ be the solutions of equations \eqref{eqv} 
for each fixed $j = \overline{1, m}$, satisfying the initial conditions
\begin{equation} \label{ic}
   C_j(0, \la) = S_j^{[1]}(0, \la) = 1, \quad C_j^{[1]}(0, \la) = S_j(0, \la) = 0.                       
\end{equation}

It is well-known, that the boundary value problem $L$ has a purely discrete spectrum,
which consists of a countable set of eigenvalues. They coincide (counting with their multiplicities)
with the zeros of the entire analytic function, called the {\it characteristic function}, which can be constructed by the following recurrent definition 
(see \cite{BS17, LP09, LY12, Yur09}).

\begin{enumerate}
\item For $m = 1$, the characteristic function $\Delta(\la)$ is defined by the following formulas
for different types of boundary conditions:
\begin{equation} \label{Dsimp}
\arraycolsep=1.4pt\def\arraystretch{1.5}
\left.
\begin{array}{ll}
y_1(0) = y_1(\pi) = 0 \colon & \quad \Delta(\la) = S_1(\pi, \la), \\
y_1(0) = y_1^{[1]}(\pi) = 0 \colon & \quad \Delta(\la) = S_1^{[1]}(\pi, \la) + \ga_1 S_1(\pi, \la), \\
y_1^{[1]}(0) = y_1(\pi) = 0 \colon & \quad \Delta(\la) = C_1(\pi, \la), \\
y_1^{[1]}(0) = y_1^{[1]}(\pi) = 0 \colon & \quad \Delta(\la) = C_1^{[1]}(\pi, \la) + \ga_1 C_1(\pi, \la).
\end{array} \qquad
\right\}
\end{equation}

\item Let $G$ has more than one edge, and let $u$ be an internal vertex of degree $r$.
Splitting the vertex $u$, we split $G$ into $r$ subtrees $G_j$, $j = \overline{1, r}$.
For each $j = \overline{1, r}$, let $\Delta_j^D(\la)$ and $\Delta_j^N(\la)$ be 
the characteristic functions for the boundary value problems for equations \eqref{eqv}
on the tree $G_j$ with the Dirichlet $y(u) = 0$ and the Neumann $y^{[1]}(u) = 0$ boundary condition
in the vertex $u$, respectively, the matching conditions \eqref{MC} in the vertices $v \in \mbox{int}\, G_j$ and
the boundary conditions BC in $v \in \partial G_j \backslash \{ u \}$. 
Then
\begin{equation} \label{Dfact}
  \Delta(\la) = \sum_{j = 1}^r \Delta_j^N(\la) \prod_{\substack{k = 1 \\ k \ne j}}^r \Delta_k^D(\la).
\end{equation}
This definition does not depend on the choice of $u$.
\end{enumerate} 

Define $\rho = \sqrt{\la}$, $\mbox{Re}\, \rho \ge 0$. Denote by $L_0$ the boundary value problem 
$L$ with $\sigma_j = 0$ and $\ga_j = 0$, $j = \overline{1, m}$, and denote by $\Delta_0(\la)$ its characteristic function.

\begin{lem} \label{lem:Delta0}
The characteristic function $\Delta_0(\la)$ can
be represented in the form
\begin{equation} \label{D0factR}
  \Delta_0(\la) = \rho^{1-d} R_m(\sin \rho \pi, \cos \rho \pi),
\end{equation}
where $d$ is the number of the Dirichlet conditions among BC
and $R_m$ is a polynomial of two variables of degree $m$. Moreover, 
$R_m$ can be represented in the form
\begin{equation} \label{RQ}
   R_m(\sin \rho \pi, \cos \rho \pi) = 
   \begin{cases} 
   \sin \rho \pi Q_{m-1}(\cos \rho \pi), \quad \text{if $d$ is even}, \\
   Q_m(\cos \rho \pi), \quad \text{if $d$ is odd},
   \end{cases}
\end{equation}
where $Q_k$, $k = m-1, m$, are polynomials of degree $k$, such that
\begin{equation} \label{symQ}
   Q_k(z) = (-1)^k Q_k(-z).
\end{equation}        
\end{lem}

\begin{proof}
For the one-edge trees with $\sigma_1 = 0$ and $\ga_1 = 0$ 
we have the following characteristic functions, which satisfy the assertion of the lemma.
\begin{align*}
y_1(0) = y_1(\pi) = 0 \colon & \quad \Delta_0(\la) = \frac{\sin \rho \pi}{\rho}, \quad d = 2,\\
y_1(0) = y_1'(\pi) = 0 \colon & \quad \Delta_0(\la) = \cos \rho \pi, \quad d = 1, \\
y_1'(0) = y_1(\pi) = 0 \colon & \quad \Delta_0(\la) = \cos \rho \pi, \quad d = 1, \\
y_1'(0) = y_1'(\pi) = 0 \colon & \quad \Delta_0(\la) = -\rho \sin \rho \pi, \quad d = 0.
\end{align*}

For $m > 1$, one can prove the lemma by induction, factorizing the characteristic function
according to \eqref{Dfact}. 
\end{proof}

Note that the degree of the polynomial $R_m$ is exactly $m$. This follows from the next lemma, which 
is also easily proved by induction.

\begin{lem}
The following asymptotic relation holds as $\rho = i \tau$, $\tau \to +\infty$:
$$
   \Delta_0(\rho^2) = \tau^{1 - d} 2^{-m} \prod_{v \in V} |E_v| \exp(m \tau \pi) (1 + o(1)).
$$
\end{lem}

\begin{lem} \label{lem:asymptla0}
The eigenvalues of the boundary value problem $L_0$ can be numbered as 
$\{ \la_{nk}^0 \}_{n \in \mathbb N_0, \, k = \overline{1, m}}$,
so that
\begin{equation} \label{asymptla0}
   \rho_{nk}^0 = \sqrt{\lambda_{nk}^0} = n + \al_k, \quad n \in \mathbb N_0, \, k = \overline{1, m},
\end{equation}
where $\mathbb N_0 = \mathbb N \cup \{ 0 \}$, $\al_k \in [0, 1)$. For every $\al_{k_1} \ne 0$, 
there exists $\al_{k_2} = 1 - \al_{k_1}$, $1 \le k_1, k_2 \le m$.

Positive eigenvalues occur in the sequence $\{ \la_{nk}^0 \}_{n \in \mathbb N, \, k = \overline{1, m}}$ the number of times,
equal to their multiplicities. The multiplicity of the zero eigenvalue equals 
$z_{\al} + 1 - d$, where $z_{\al}$ is the number of zeros among $\al_k$, $k = \overline{1, m}$.
\end{lem}                                                                                         

\begin{proof}
In view of \eqref{RQ} and \eqref{symQ}, the function $R(\rho) := R_m(\sin \rho \pi, \cos \rho \pi)$
has the period $1$. For definiteness, consider the case of even $d$.
Then the roots of $R(\rho)$ equal $n + \al_k$ for $n \in \mathbb Z$ and 
$\al_k = \frac{1}{\pi} \arccos z_k$, $k = \overline{1, m-1}$,
where $z_k$ are the roots of the polynomial $Q_{m-1}(z)$ (counting with their 
multiplicities), and $\al_m = 0$.
Since the problem $L$ is self-adjoint, its eigenvalues are real. Hence $z_k \in [-1, 1]$.
In view of \eqref{symQ}, a root $z_{k_1} = 1$ always has a pair $z_{k_2} = -1$, $k_1, k_2 < m$.
In this case, we put $\al_{k_1} = \al_{k_2} = 0$, and consequently, $\al_k \in [0, 1)$
for $k =  \overline{1, m}$. 
According to \eqref{RQ} and \eqref{symQ}, the function $R(\rho)$ is either even or odd. Hence
$$
   R(1 - \al_k) = R(-1 + \al_k) = R(\al_k) = 0.
$$
Consequently, for each $k_1 = \overline{1, m}$, such that $\al_{k_1} \ne 0$,
there exists $k_2 = \overline{1, m}$, such that $1 - \al_{k_1} = \al_{k_2}$. Using \eqref{D0factR}, we arrive at formula \eqref{asymptla0} for the zeros of $\Delta_0(\la)$.

One can also easily calculate the multiplicity of the zero
eigenvalue, using Lemma~\ref{lem:Delta0}.                         
\end{proof}

Denote by $B_{2, a}$ the class of Paley-Wiener functions of exponential type not greater than $a$,
belonging to $L_2(\mathbb R)$. We use the notations $\varkappa_{k, odd}(\rho)$ and $\varkappa_{k, even}(\rho)$
for various odd and even functions from $B_{2, k\pi}$, respectively. Clearly, that
$$
   \varkappa_{k, odd}(\rho) = \int_0^{k \pi} \mathscr K(t) \sin \rho t \, dt, \quad
   \varkappa_{k, even}(\rho) = \int_0^{k \pi} \mathscr N(t) \cos \rho t \, dt,
$$
where $\mathscr K, \mathscr N \in L_2(0, k\pi)$.

Using the transformation operators \cite{HM04-transform}, one can obtain the following relations (see \cite{HM03, HM04-2spectra}):
\begin{equation} \label{CS}
\arraycolsep=1.4pt\def\arraystretch{1.5}
\left.
\begin{array}{ll}
   C_j(\pi, \la) & = \cos \rho \pi + \varkappa_{1, even}(\rho), \quad 
   C_j^{[1]}(\pi, \la) = -\rho \sin \rho \pi + \rho \varkappa_{1, odd}(\rho) + C_j^{[1]}(\pi, 0), \\
   S_j(\pi, \la) & = \dfrac{\sin \rho \pi}{\rho} + \dfrac{\varkappa_{1, odd}(\rho)}{\rho}, \quad
   S_j^{[1]}(\pi, \la) = \cos \rho \pi + \varkappa_{1, even}(\rho).
\end{array}
\qquad \right\}
\end{equation}

\begin{lem} \label{lem:Delta}
The characteristic function $\Delta(\la)$ admits the following representation
\begin{equation} \label{DfactR}
   \Delta(\la) = \rho^{1-d} \left( R_m(\sin \rho \pi, \cos \rho \pi) + \varkappa_m(\rho) + \frac{C}{\rho}\right),
\end{equation}
where $d$ and $R_m(\sin \rho \pi, \cos \rho \pi)$ are the same as in Lemma~\ref{lem:Delta0},
$C$ is a constant, and
\begin{equation} \label{kappam}
   \varkappa_m(\rho) = \begin{cases}
                          \varkappa_{m, odd}(\rho), \quad \text{if $d$ is even}, \\
                          \varkappa_{m, even}(\rho), \quad \text{if $d$ is odd}.
                       \end{cases}
\end{equation}
If $d > 0$, then $C = 0$.
\end{lem}

\begin{proof}
Relation \eqref{DfactR} is easily proved by induction, using \eqref{Dsimp}, \eqref{Dfact} and \eqref{CS}.
Here we focus on the proof of the fact, that $C = 0$ for $d > 0$.
For $m = 1$ this fact follows from \eqref{CS}. Otherwise it is convenient to use the 
representation for the characteristic function $\Delta(\la)$, described below.

Suppose the relation $C = 0$ is proved for trees with less than $m$ edges of arbitrary structure with all the possible combinations
of Dirichlet and Neumann boundary conditions with $d > 0$. Consider a tree $G$ with $m$ edges.
Let $e_j$ be a boundary edge, incident to a boundary vertex $b$ and an internal vertex $u$.
Suppose that $b$ corresponds to $x_j = 0$ and the boundary condition in $b$ is $y(b) = 0$.
Then we have
\begin{equation} \label{DKP}
   \Delta(\la) = S_j(\pi, \la) \Delta^K(\la) + S_j^{[1]}(\pi, \la) \Delta^{\Pi}(\la),
\end{equation}
where $\Delta^K(\la)$ is the characteristic function of the tree $G$ without the edge $e_j$
with the standard matching conditions \eqref{MC} in the vertices $v \in \mbox{int}\, G$ and the conditions BC
in all $v \in \partial G \backslash \{ b \}$, and
$$
   \Delta^{\Pi}(\la) = \prod_{e_k \in E_v \backslash \{ e_j \}} \Delta_k^D(\la),
$$
where $\Delta_k^D(\la)$ are the characteristic functions of the subgraphs $G_k$, obtained from $G$
by splitting it at the vertex $u$, with the Dirichlet boundary conditions in $u$.
Each edge $e_k$ belongs to $G_k$.
Note that for the functions $\Delta_k^D(\la)$ the relation \eqref{DfactR} holds with $C = 0$
by the induction hypothesis. Thus, the second term in \eqref{DKP} does not contain $C$.
In view of \eqref{CS} and \eqref{DfactR}, the first term of \eqref{DKP} has the form
\begin{multline*}
   S_j(\pi, \la) \Delta^K(\rho) = \rho^{1-d} (\sin \rho \pi + \varkappa_{1, odd}(\rho))
   \left( R_{m-1}(\sin \rho \pi, \cos \rho \pi) + \varkappa_{m-1}(\rho) + \frac{C}{\rho} \right)
   \\ = \rho^{1-d} R_m(\sin \rho \pi, \cos \rho \pi) + \varkappa_m(\rho), 
\end{multline*}
where $\varkappa_{m-1}(\rho)$ and $\varkappa_m(\rho)$ can be either odd or even.
Thus, we get from \eqref{DKP}, that $C = 0$ for $\Delta(\la)$, represented in  the form \eqref{DfactR}
for an arbitrary tree $G$.

\end{proof}

Below we use the same symbol $\{ \varkappa_n \}$ for various sequences from $l_2$.
Lemmas~\ref{lem:asymptla0} and \ref{lem:Delta} yield the following result.
                                                 
\begin{lem} \label{lem:asymptla}
The eigenvalues of the boundary value problem $L$ can be numbered as 
$\{ \la_{nk} \}_{n \in \mathbb N_0, \, k = \overline{1, m}}$, in such a way that
\begin{equation} \label{asymptla1}
    \rho_{nk} = \sqrt{\la_{nk}} = \rho_{nk}^0 + o(1) = n + \al_k + o(1), \quad n \to \iy, \: k = \overline{1, m}.
\end{equation}

The numbers $\al_k$, $k = \overline{1, m}$, satisfy the properties stated in Lemma~\ref{lem:asymptla0}.
Moreover, if for a fixed $k = \overline{1, m}$ we have $\al_k \ne \al_j$ for all
$j = \overline{1, m} \backslash \{k\}$, then 
\begin{equation} \label{asymptla2}
   \rho_{nk} = n + \al_k + \varkappa_n, \quad n \in \mathbb N.
\end{equation}

The numbers of occurrences of the values in the sequence $\{ \la_{nk} \}$ coincide with the
multiplicities of the corresponding eigenvalues for $n \in \mathbb N$ and for $n = 0$, $\al_k \ne 0$.
The number of the remaining eigenvalues equals $z_{\al} + 1 - d$.
\end{lem}

\begin{proof}
Asymptotic formula \eqref{asymptla1} can be obtained from \eqref{DfactR} 
by the standard technique, based on Rouch\'e's
theorem (see, for example, \cite[Theorem~1.1.3]{FY01}). 

Fix $k = \overline{1, m}$ such that $\al_k \ne \al_j$, 
$j = \overline{1, m} \backslash \{k\}$.
Substituting 
$$
   \rho = \rho_{nk} = n + \al_k + \eps_n, \quad \eps_n = o(1),
$$
into \eqref{DfactR}, we get
$$
   \Delta(\la_{nk}) = \rho_{nk}^{1-d}(R(n + \al_k + \eps_n) + \varkappa_n) = 0, \quad n \in \mathbb N_0,
$$
where $R(\rho) := R_m(\sin \rho \pi, \cos \rho \pi)$. This implies
$R(\al_k + \eps_n) = \varkappa_n$, $n \in \mathbb N_0$.
Since the function $R(\rho)$ is analytic in the neighborhood of $\al_k$, the following expansion is valid
$$
    R(\al_k + z) = R(\al_k) + R'(\al_k) z + O(z^2), \quad |z| \to 0.
$$
Note that $R(\al_k) = 0$ and $R'(\al_k) \ne 0$, since $\al_k$ is a simple root of $R(\rho)$.
Thus, we obtain $\eps_n = \varkappa_n$, $n \in \mathbb N_0$.
\end{proof}

\bigskip

{\bf \large 3. Partial inverse problem}

\bigskip

We suppose that the tree $G$ is divided into two connected subtrees $G_{known}$ and $G_{unknown}$,
having the common vertex $w \in \mbox{int} \, G$ (see Figure~\ref{img:1}). Denote their sets of edges
by $E_{known}$ and $E_{unknown}$, respectively. Suppose that we are given the functions
$\sigma_j$ for all $e_j \in E_{known}$ and would like to recover $\sigma_j$
for all $e_j \in E_{unknown}$. The constants $\ga_j$, $j = \overline{1, m}$, are supposed to be known. 
We shall give a rigorous formulation of the partial inverse problem later in this section (see Problem~\ref{prob:pip}), but before we should define the spectral data for recovering the potentials on the edges $E_{unknown}$.

\begin{figure}[h!]
\centering
\begin{tikzpicture}
\filldraw (0, 0) circle (2pt) node[anchor=south]{$w$};
\draw[dashed] (0, 0) edge [bend left] (1, 2);
\draw[dashed] (1, 2) edge [bend left] (4, 0.5);
\draw[dashed] (0, 0) edge [bend right] (0.1, -3);
\draw[dashed] (0.1, -3) edge [bend right] (4, -2);
\draw[dashed] (4, -2) arc(-35:35:2.17);
\draw (2.8, -1.8) node{$G_{unknown}$};
\filldraw (1, 1) circle (1pt);
\draw (0, 0) edge (1, 1);
\draw (0.6, 0.8) node{$G_1$};
\draw[dashed] (0, 0) edge [bend left] (0.9, 1.4);
\draw[dashed] (0, 0) edge [bend right] (1.5, 1);
\draw[dashed] (1.5, 1) arc (-20:120:0.35);
\filldraw (1.5, -0.3) circle (1pt);
\draw (0, 0) edge (1.5, -0.3);
\draw (1.4, 0) node{$G_2$}; 
\filldraw (2.5, -0.8) circle (1pt);
\draw (1.5, -0.3) edge (2.5, -0.8);
\filldraw (3, 0) circle (1pt);
\draw (1.5, -0.3) edge (3, 0);
\filldraw (3.3, -0.5) circle (1pt);
\draw (2.5, -0.8) edge (3.3, -0.5);
\filldraw (3.6, -0.8) circle (1pt);
\draw (2.5, -0.8) edge (3.6, -0.8);
\filldraw (3.3, -1.1) circle (1pt);
\draw (2.5, -0.8) edge (3.3, -1.1);
\draw[dashed] (0, 0) edge [bend left] (3.5, -0.1);
\draw[dashed] (0, 0) edge [bend right] (3.5, -1.3);
\draw[dashed] (3.5, -1.3) arc (-80:80:0.6);	
\filldraw (0.5, -1) circle (1pt);
\draw (0, 0) edge (0.5, -1);
\filldraw (1.3, -2) circle (1pt);
\draw (0.5, -1) edge (1.3, -2);
\filldraw (0.7, -2) circle (1pt);
\draw (0.5, -1) edge (0.7, -2);
\draw[dashed] (0, 0) edge [bend left] (1.7, -2.5);
\draw[dashed] (0, 0) edge [bend right] (0.3, -2.5);
\draw[dashed] (0.3, -2.5) arc (210:330:0.8);
\draw (0.2, -1) node{$G_3$};
\filldraw (-1.5, -0.5) circle (1pt);
\draw (0, 0) edge (-1.5, -0.5);
\filldraw (-1.5, 0.5) circle (1pt);
\draw (0, 0) edge (-1.5, 0.5);
\filldraw (-2.5, 1) circle (1pt);
\draw (-1.5, 0.5) edge (-2.5, 1);
\filldraw (-2.8, 0.5) circle (1pt);
\draw (-1.5, 0.5) edge (-2.8, 0.5);
\filldraw (-2.5, 0) circle (1pt);
\draw (-1.5, 0.5) edge (-2.5, 0);
\filldraw (-2.7, -0.5) circle (1pt);
\draw (-1.5, -0.5) edge (-2.7, -0.5);
\filldraw (-2.5, -1) circle (1pt);
\draw (-1.5, -0.5) edge (-2.5, -1);
\filldraw (-3.5, 0) circle (1pt);
\draw (-2.5, 0) edge (-3.5, 0);
\filldraw (-3.7, 1.3) circle (1pt);
\draw (-2.5, 1) edge (-3.7, 1.3);
\filldraw (-3.5, 0.8) circle (1pt);
\draw (-2.5, 1) edge (-3.5, 0.8);
\filldraw (-4.5, 0.3) circle (1pt);
\draw (-3.5, 0) edge (-4.5, 0.3);
\filldraw (-4.5, -0.3) circle (1pt);
\draw (-3.5, 0) edge (-4.5, -0.3);
\draw[dashed] (0, 0) edge [bend left] (-4, -1.5);
\draw[dashed] (0, 0) edge [bend right] (-4, 1.5);
\draw[dashed] (-4, 1.5) arc(100:260:1.5);
\draw (-3.7, -1) node{$G_{known}$};
\end{tikzpicture}
\caption{Graph $G$}
\label{img:1}
\end{figure}

Denote by $\Delta_{known}^K(\la)$ and $\Delta_{unknown}^K(\la)$ the characteristic functions
for the subtrees $G_{known}$ and $G_{unknown}$, respectively. We suppose that
for the both subtrees, there are standard matching conditions \eqref{MC} in all $v \in \mbox{int}\, G$
and the conditions BC in all $v \in \partial G$. 
Let us split the vertex $w$ of the subtree $G_{unknown}$, and get the subtrees $G_j$, 
$j = \overline{1, p}$. Let $\Delta_j^D(\la)$ be the characteristic functions for $G_j$
with the Dirichlet boundary condition in $w$, conditions \eqref{MC} in all $v \in \mbox{int}\, G_j$ and BC in all $v \in \partial G_j \backslash \{ w \}$.  Define
\begin{equation} \label{DPiprod}
   \Delta_{unknown}^{\Pi}(\la) = \prod_{j = 1}^p \Delta_j^D(\la),
\end{equation}
and similarly define $\Delta_{known}^{\Pi}(\la)$.
In view of the introduced notations, relation \eqref{DfactR} implies
\begin{equation} \label{DeltaPiK}
   \Delta(\la) = \Delta_{known}^K (\la) \Delta_{unknown}^{\Pi}(\la) 
   + \Delta_{known}^{\Pi}(\la) \Delta_{unknown}^K (\la).
\end{equation}

Let $l$, $l_j$, $j = \overline{1, p}$, be the numbers of edges and $r$, $r_j$, $j = \overline{1,p}$,
be the numbers of the Dirichlet boundary conditions for the subtrees $G_{unknown}$ and
$G_j$, $j = \overline{1, p}$, respectively. Obviously,
\begin{equation} \label{sumlr}
   l = \sum_{j = 1}^p l_j, \quad r + p = \sum_{j = 1}^p r_j.
\end{equation}

For definiteness, consider the case of odd $r$. 
The cases of positive even $r$ and of $r = 0$ can be studied similarly.
By Lemma~\ref{lem:Delta} we have
\begin{equation} \label{DK}
   \Delta_{unknown}^K(\la) = \rho^{1-r} \left( R_l^K(\sin \rho \pi, \cos \rho \pi) + \int_0^{l\pi} N(t) \cos \rho t \, dt \right), \quad N \in L_2(0, l\pi),
\end{equation} 
\begin{equation} \label{DjD}
   \Delta_j^D(\la) = \rho^{1 - r_j} (R_{l_j,j}(\sin \rho \pi, \cos \rho \pi) + \varkappa_{l_j,j}(\rho)),
\end{equation}
where $R_l^K$ and $R_{l_j, j}$, $j = \overline{1, p}$, 
are the polynomials of degrees $l$ and $l_j$, $j = \overline{1, p}$, respectively,
of the form described in Lemma~\ref{lem:Delta0}, and
$$
    \varkappa_{l_j, j}(\rho) = \begin{cases}
                                   \varkappa_{l_j, odd}(\rho), \quad \text{if $r_j$ is even}, \\
                                   \varkappa_{l_j, even}(\rho), \quad \text{if $r_j$ is odd}.
                               \end{cases}
$$
Substituting \eqref{DjD} into \eqref{DPiprod} and using \eqref{sumlr}, we obtain
\begin{equation} \label{DPi}
   \Delta_{unknown}^{\Pi}(\la) = \rho^{-r} \left(R_l^{\Pi}(\sin \rho \pi, \cos \rho \pi) + 
   \int_0^{l \pi} K(t) \sin \rho t \, dt \right), \quad K \in L_2(0, l\pi),
\end{equation}
where 
$$
   R_l^{\Pi}(\sin \rho \pi, \cos \rho \pi) = \prod_{j = 1}^p R_{l_j, j}(\sin \rho \pi, \cos \rho \pi).
$$

Denote by $\Lambda(L)$ the spectrum of the problem $L$. For simplicity, we assume that all the eigenvalues in $\Lambda(L)$ are positive.
This condition can be achieved by a shift of the spectrum: $\la \mapsto \la + C$, $\sigma_j \mapsto \sigma_j + C x_j$, $j = \overline{1, m}$.

Substitute relations \eqref{DK} and \eqref{DPi} into \eqref{DeltaPiK} for $\la \in \Lambda(L)$:
\begin{multline*}
    \Delta_{known}^K (\la) \left( R_l^{\Pi}(\sin \rho \pi, \cos \rho \pi) + \int_0^{l \pi} K(t) \sin \rho t \, dt  \right)
    \\ + \Delta_{known}^{\Pi}(\la) \left( \rho R_l^K (\sin \rho \pi, \cos \rho \pi) + \int_0^{l \pi} N(t) \cos \rho t \, dt \right) = 0. 
\end{multline*}
Rewrite this relation in the following form
\begin{gather} \label{main}
    \Delta_{known}^K(\la) \int_0^{l \pi} K(t) \frac{\sin \rho t}{\rho} \, dt + \Delta_{known}^{\Pi}(\la) \int_0^{l \pi} N(t) \cos \rho t \, dt = g(\la), \quad \la \in \Lambda(L), \\
    \label{defg}
    g(\la) := -\frac{1}{\rho}\Delta_{known}^K(\la) R_l^{\Pi}(\sin \rho \pi, \cos \rho \pi) - \Delta_{known}^{\Pi}(\la) R_l^K(\sin \rho \pi, \cos \rho \pi).
\end{gather}

Consider the vector functions
\begin{equation} \label{defs}
   f(t) = \begin{bmatrix}
               K(t) \\ N(t) 
          \end{bmatrix}, \quad  
   s(t, \la) = \begin{bmatrix}
                  \Delta_{known}^K(\la) \rho^{-1} \sin \rho t, \\ \Delta_{known}^{\Pi}(\la) \cos \rho t
               \end{bmatrix}, \quad \rho > 0,
\end{equation}
belonging to the real Hilbert space $\mathcal{H} := L_2(0, l\pi) \oplus L_2(0, l\pi)$ as functions of $t$.
The scalar product and the norm in $\mathcal H$ are defined as follows
$$
   (g, h)_{\mathcal H} = \int_0^{l\pi} ( g_1(t) h_1(t) + g_2(t) h_2(t)) \, dt, \quad
   \| g \|_{\mathcal H} = \sqrt{\int_0^{l\pi} (g_1^2(t) + g_2^2(t)) \, dt}, 
$$
$$
   g = \begin{bmatrix} g_1 \\ g_2 \end{bmatrix}, \quad h = \begin{bmatrix} h_1 \\ h_2 \end{bmatrix}, \quad g, h \in \mathcal H.    
$$
Then relation \eqref{main} takes the form
\begin{equation} \label{scal}
    (f(t), s(t, \la))_{\mathcal H} = g(\la).
\end{equation}

Using \eqref{DK} and \eqref{DPi}, one can derive additional relations in the form similar to \eqref{scal}. Indeed, since the characteristic functions $\Delta_{unknown}^K(\la)$, $\Delta_{unknown}^{\Pi}(\la)$ and their analogs for the boundary value problem $L_0$ with the zero potentials are entire, we get that the functions
\begin{align*}
\rho^{1-r} \int_0^{l\pi} N(t) \cos \rho t \, dt & = \rho^{1-r} \int_0^{l\pi} N(t) \left( \sum_{j = 0}^{\iy} \frac{(-1)^j (\rho t)^{2 j}}{(2 j)!}\right) \, dt \\
\rho^{-r} \int_0^{l\pi} K(t) \sin \rho t \, dt & = \rho^{-r} \int_0^{l\pi} K(t) \left( \sum_{j = 0}^{\iy} \frac{(-1)^j (\rho t)^{2 j + 1}}{(2 j + 1)!} \right) \, dt 
\end{align*}
are analytic at zero. 
Consequently,
\begin{equation} \label{cond0}
\int_0^{l\pi} N(t) t^{2 j} \, dt = 0, \quad \int_0^{l\pi} K(t) t^{2 j + 1} \, dt = 0, \quad j = \overline{0, (r-3)/2}.
\end{equation}
Define the vector functions
\begin{equation} \label{defsj}
s_j(t) = \begin{bmatrix} 0 \\ t^j \end{bmatrix}, \quad \text{if $j$ is even}, \qquad 
s_j(t) = \begin{bmatrix} t^j \\ 0 \end{bmatrix}, \quad \text{if $j$ is odd},
\end{equation}
and rewrite equations \eqref{cond0} in the form
\begin{equation} \label{scal0}
(f(t), s_j(t))_{\mathcal H} = 0, \quad j = \overline{0, r-2}.
\end{equation}

Denote by $\Lambda'(L)$ an arbitrary subset of $\Lambda(L)$, such that the sequence
\begin{equation} \label{defS}
\mathcal S := \{ s(t, \la) \}_{\la \in \Lambda'(L)} \cup \{ s_j(t) \}_{j = 0}^{r-2}
\end{equation}
is complete in $\mathcal H$. Note that
if a subset $\Lambda'(L)$ exists, it is countable. The following auxiliary problem plays a crucial role in the further investigation.

\begin{prob} \label{prob:ap}
Given the functions $\sigma_j$ for all $e_j \in E_{known}$ and 
the set $\Lambda'(L)$, construct the characteristic functions $\Delta_{unknown}^K(\la)$ and $\Delta_{unknown}^{\Pi}(\la)$.
\end{prob}

Along with $L$ we consider the boundary value problem $\tilde L$ of the same form,
but with different coefficients $\tilde \sigma_j$, $j = \overline{1, m}$. 
We agree that if a certain symbol $\ga$ denotes an object related to $L$,
the symbol $\tilde \ga$ with tilde denotes the similar object related to $\tilde L$.

The following theorem asserts the uniqueness for the solution of Problem~\ref{prob:ap}.

\begin{thm} \label{thm:uniqap}
Suppose that $\sigma_j = \tilde \sigma_j$ for all $e_j \in E_{known}$ and 
there exist subsets of the spectra $\Lambda'(L)$ and $\Lambda'(\tilde L)$, such that 
$\Lambda'(L) = \Lambda'(\tilde L)$. Then $\Delta_{unknown}^K(\la) \equiv \tilde \Delta_{unknown}^K(\la)$
and $\Delta_{unknown}^{\Pi}(\la) \equiv \tilde \Delta_{unknown}^{\Pi}(\la)$, $\la \in \mathbb C$.
\end{thm}

\begin{proof}
Since $\sigma_j = \tilde \sigma_j$ for $e_j \in E_{known}$, we have 
$$
\Delta_{known}^K(\la) \equiv \tilde \Delta_{known}^K(\la), \quad 
\Delta_{known}^{\Pi}(\la) \equiv \tilde \Delta_{known}^{\Pi}(\la).
$$ 
Thus, in view definitions \eqref{defg} and \eqref{defs},
we have $g(\la) = \tilde g(\la)$, $s(t, \la) = \tilde s(t, \la)$, $\la \in \Lambda'(L)$.
Since the system $\mathcal S$ is complete in $\mathcal H$,
relations \eqref{scal} and \eqref{scal0} yield $f = \tilde f$, i.e. $K = \tilde K$ and $N = \tilde N$
in $L_2(0, l\pi)$. Using \eqref{DK} and \eqref{DPi}, we arrive at the assertion of the theorem.
\end{proof}

If the system $\mathcal S$ is 
an unconditional basis in $\mathcal H$, i.e. the normalized system
$$
\{ s(t, \la) / \| s(t, \la) \|_{\mathcal H} \}_{\la \in \Lambda'(L)} \cup \{ s_j(t) / \| s_j(t) \|_{\mathcal H} \}_{j = 0}^{r-2}
$$
is a Riesz basis, one can use the following algorithm for the solution of Problem~\ref{prob:ap}.

\begin{alg} \label{alg:ap}
Let $\sigma_j$ for $e_j \in E_{known}$ and the set $\Lambda'(L)$ be given.

\begin{enumerate}
\item Using $\sigma_j$ for $e_j \in E_{known}$, construct the characteristic 
functions $\Delta_{known}^K(\la)$ and $\Delta_{known}^{\Pi}(\la)$.

\item Construct the trigonometric polynomials by formulas
$$
R_l^K(\sin \rho \pi, \cos \rho \pi) = \rho^{r-1} \Delta_{unknown, 0}^K(\la), \quad
R_l^{\Pi}(\sin \rho \pi, \cos \rho \pi) = \rho^r \Delta_{unknown, 0}^{\Pi}(\la),
$$
where $\Delta_{unknown, 0}^K(\la)$ and $\Delta_{unknown, 0}^{\Pi}(\la)$ are the characteristic functions for the graph $G_{unknown}$ with the zero potential.

\item Construct the numbers $g(\la)$ and the vector functions $s(t, \la)$ for $\la \in \Lambda'(L)$ by 
\eqref{defg}, \eqref{defs}, and the vector functions $s_j(t)$ by \eqref{defsj}.

\item Recover the vector function $f(t)$ from its coordinates with respect to the Riesz basis
(see relations \eqref{scal}, \eqref{scal0}), i.e. find $K(t)$, $N(t)$.

\item Find the functions $\Delta_{unknown}^K(\la)$ and $\Delta_{unknown}^{\Pi}(\la)$, using \eqref{DK} and \eqref{DPi}.
\end{enumerate}
\end{alg}

Now we describe, how the functions $\sigma_j$ on the tree $G_{unknown}$ can be recovered.
Put $\partial G_{unknown} \backslash \{ w \} =: \{ v_k \}_{k = 1}^b$. For each $k = \overline{1, b}$,
let $L_k$ be the boundary value problem for the system of equations \eqref{eqv} on the graph $G$ with the standard matching conditions in the internal vertices
$v \in \mbox{int} \, G$ and the conditions BC in the vertices $v \in \partial G \backslash \{ v_k \}$.
In the vertex $v_k$, the boundary condition is different from the one in BC, i.e. if the problem $L$ has 
the Dirichlet condition $y(v_k) = 0$, then the problem $L_k$ has the Neumann one: $y^{[1]}(v_k) = 0$, and vice versa.

Fix $k = \overline{1, b}$.
Problem~\ref{prob:ap} can be formulated and solved for the boundary value problems $L_k$ with necessary modifications.
Namely, we supposed above for definiteness, that the number of the boundary vertices among $\{ v_k \}_{k = 1}^b$
with the Dirichlet boundary conditions (denoted by $r$) is odd. For the problems $L_k$, this number is even.
If it is positive, repeating the arguments above, we obtain the following vector functions
instead of $ s(t, \la)$ and $s_j(t, \la)$:
$$
   s_k(t, \la) = \begin{bmatrix}
                  \Delta_{known}^K(\la) \cos \rho t, \\ \Delta_{known}^{\Pi}(\la) \rho \sin \rho t
               \end{bmatrix}, \quad
  s_{jk}(t) = \begin{bmatrix} t^j \\ 0 \end{bmatrix}, \: \text{if $j$ is even}, \quad
  s_{jk}(t) = \begin{bmatrix} 0 \\ t^j \end{bmatrix}, \: \text{if $j$ is odd}.	
$$ 

Let $\Lambda'(L_k)$ be a subset of the spectrum $\Lambda(L_k)$, such that the system of vector functions
$$
\mathcal S_k := \{ s_k(t, \la) \}_{\la \in \Lambda'(L_k)} \cup \{ s_{jk}(t) \}_{j = 0}^{r_k - 2}
$$
is complete in $\mathcal H$. Denote by $\Delta_{unknown, k}^K(\la)$ and
$\Delta_{unknown, k}^{\Pi}(\la)$ the analogs of the characteristic functions
$\Delta_{unknown}^K(\la)$ and $\Delta_{unknown}^{\Pi}(\la)$ for the problem $L_k$.
Clearly, $\Delta_{unknown, k}^K(\la)$ and $\Delta_{unknown, k}^{\Pi}(\la)$ are uniquely specified by $\sigma_j$ for $e_j \in E_{known}$ and the subspectrum $\Lambda'(L_k)$. If $\mathcal S_k$
is an unconditional basis in $\mathcal H$, $\Delta_{unknown, k}^K(\la)$ and $\Delta_{unknown, k}^{\Pi}(\la)$ can be constructed
by Algorithm~\ref{alg:ap} with necessary modifications.

The following problem is, in fact, the complete inverse problem by the system of spectra for the tree $G_{unknown}$, if $p > 1$. 

\begin{prob} \label{prob:ip}
Given the characteristic functions $\Delta_{unknown}^K(\la)$, $\Delta_{unknown, k}^K(\la)$, $k = \overline{1, b-1}$,
find the potentials $\sigma_j$ for $e_j \in E_{unknown}$.
\end{prob}

Note that the number of required characteristic functions $\Delta_{unknown, k}^K(\la)$ is one less than the number of the boundary vertices in $G_{unknown}$.
Problem~\ref{prob:ip} has been studied in \cite{Yur05, Yur06, FIY08}, where the uniqueness theorem 
and the constructive algorithm for its solution, based on the method of spectral mappings \cite{FY01, Yur02}, have been obtained.
Combining these methods with the ones developed for Problem~\ref{prob:ap}, we arrive at the solution of the following partial inverse problem.

\begin{prob} \label{prob:pip}
Given the potentials $\sigma_j$ for $e_j \in E_{known}$ and the subspectra $\Lambda'(L)$, $\Lambda'(L_k)$, $k = \overline{1, b-1}$,
find $\sigma_j$ for $e_j \in E_{unknown}$.
\end{prob}

Theorem~\ref{thm:uniqap} together with the uniqueness of the solution of Problem~\ref{prob:ip}, imply the following uniqueness theorem for Problem~\ref{prob:pip}.

\begin{thm} \label{thm:uniq}
Suppose that $\sigma_j = \tilde \sigma_j$ for $e_j \in E_{known}$ and there exist subspectra
$\Lambda'(L)$, $\Lambda'(\tilde L)$, $\Lambda'(L_k)$, $\Lambda'(\tilde L_k)$, $k = \overline{1, b-1}$,
such that $\Lambda'(L) = \Lambda'(\tilde L)$, $\Lambda'(L_k) = \Lambda'(\tilde L_k)$, $k = \overline{1, b-1}$.
Then $\sigma_j = \tilde \sigma_j$ for $e_j \in E_{unknown}$ in $L_2(0, \pi)$.
\end{thm}

If $p > 1$ and the systems $\mathcal S$, $\mathcal S_k$, $k = \overline{1, b-1}$,
are unconditional bases in $\mathcal H$, i.e. the corresponding normalized systems are Riesz bases, one can solve Problem~\ref{prob:pip}, using the following algorithm.

\begin{alg} \label{alg:pip}
Let the functions $\sigma_j$ for $e_j \in E_{known}$ and the subspectra $\Lambda'(L)$, $\Lambda'_k(L)$, $k = \overline{1, b-1}$, be given.

\begin{enumerate}
\item Find the characteristic functions $\Delta_{unknown}^K(\la)$, $\Delta_{unknown, k}^K(\la)$, $k = \overline{1, b-1}$, 
using Algorithm~\ref{alg:ap}.

\item Recover the potentials $\sigma_j$ for $e_j \in E_{unknown}$, solving Problem~\ref{prob:ip} by the methods from \cite{Yur05, Yur06, FIY08}.
\end{enumerate}
\end{alg}

In fact, relying only on the methods of the works \cite{Yur05, Yur06, FIY08}, one can easily construct the potentials on the 
subtree $G_{unknown}$, using the full spectra $\Lambda(L)$, $\Lambda(L_k)$, $k = \overline{1, b-1}$. The principal novelty of the present paper is
that we can reduce the given data to some subspectra, given the potentials on the part $G_{known}$ a priori. 
This idea will be revealed in more details in the next sections.

\begin{remark} \label{rem:1edge}
Consider the case $p = 1$, when there is the only edge (denote it by $e_1$), incident to the vertex $w$ in the subtree $G_{unknown}$.
Then, solving Problem~\ref{prob:ap} by Algorithm~\ref{alg:ap}, we can construct the so-called Weyl function 
$$
M(\la) := \frac{\Delta_{unknown}^K(\la)}{\Delta_{unknown}^{\Pi}(\la)},
$$
for the graph $G_{unknown}$, associated with the vertex $w$. Solving the local inverse problem (see \cite[Section 2]{Yur05}, \cite[Section 4]{FIY08}),
we find the potential on the edge $e_1$, using $M(\la)$. Consequently, we can move the edge $e_1$ from $G_{unknown}$ to $G_{known}$, and reduce
the number of required eigenvalues from the subspectra $\Lambda(L_k)$, $k = \overline{1, b-1}$. If the graph $G_{unknown}$ consists of the only edge $e_1$,
we can recover $\sigma_1$ as a solution of the classical inverse problem on a finite interval by the Weyl function $M(\la)$ (see \cite{FY01}).
\end{remark}

\bigskip

{\large \bf 4. Completeness conditions}

\bigskip

In the previous section, uniqueness Theorems~\ref{thm:uniqap} and \ref{thm:uniq} are formulated for an abstract subset of the spectrum $\Lambda'(L)$ and require the completeness of the system of vector functions $\mathcal S$, defined in $\eqref{defS}$.
In this section, we provide sufficient conditions on the subspectrum $\Lambda'(L)$, that 
make the system $\mathcal S$ complete in $\mathcal H$. Our analysis is based on the asymptotic behavior of the eigenvalues, described by Lemma~\ref{lem:asymptla}.

Denote by $\mathcal K$ the subset of all indices $k = \overline{1, m}$, such that
$\al_k \in \left( 0, \frac{1}{2} \right)$ and $\al_k \ne \al_j$ for $j = \overline{1, m} \backslash \{ k \}$.
Suppose that $\mathcal K$ is nonempty. Then $\Lambda(L)$ includes the subsequences of eigenvalues in the form
\begin{equation} \label{asymptla3}
    \la_{nk} = \rho_{nk}^2, \quad \rho_{nk} = n + \al_k + \varkappa_n, \quad n \in \mathbb Z, \: k \in \mathcal K.
\end{equation}
The numeration here is different from the one used in Lemmas~\ref{lem:asymptla0} and \ref{lem:asymptla}.
In fact, each pair of subsequences from Lemma~\ref{lem:asymptla} with $\al_{k_1}$ and $\al_{k_2} = 1-\al_{k_1}$ is joined together.
Further we suppose that the eigenvalues are numbered according to \eqref{asymptla3}.
Note that this numeration is not unique. In fact, any finite number of the first values in \eqref{asymptla3} can be chosen arbitrarily 
from $\Lambda(L)$.

Recall that $l$ is the number of edges in the tree $G_{unknown}$. Assume that $l \le |\mathcal K|$. 
For definiteness, let the indices $\{ 1, 2, \dots, l \}$ belong to $\mathcal K$. 
Put $\Lambda'(L) := \{ \la_{nk} \}_{n \in \mathbb Z, \, k = \overline{1, l}} \backslash \{ \la_{n1} \}_{n = 1}^{r-1}$, and impose the following assumptions.

\smallskip

($A_1$) The eigenvalues in $\Lambda'(L)$ are distinct and positive.

($A_2$) $\Delta_{known}^{\Pi}(\la) \ne 0$, $\Delta_{unknown}^{\Pi}(\la) \ne 0$ for all $\la \in \Lambda'(L)$.

($A_3$) $\Delta_{unknown}^K(0) \ne 0$, $\Delta_{unknown}^{\Pi}(0) \ne 0$.

($A_4$) The characteristic functions $\Delta_{unknown}^K(\la)$ and $\Delta_{unknown}^{\Pi}(\la)$ do not have common zeros.

\smallskip

Note that it is possible to work with multiple eigenvalues with some technical modifications (see \cite{Bond17-1}). Assumption ($A_3$) can be easily achieved by a shift of the spectrum.
Assumption ($A_4$) holds automatically, if the graph $G_{unknown}$ has an only edge. Of course, one can also include the subsequences of eigenvalues with asymptotics $(n + \varkappa_n)^2$ or $(n + 1/2 + \varkappa_n)^2$ into $\Lambda'(L)$, but these cases require technical 
modifications, so we confine ourselves to using subsequences in the form \eqref{asymptla3}.

Further we need the following auxiliary fact (see \cite[Corollary~3]{Bond17-preprint}).

\begin{lem} \label{lem:prod}
Let $\{ \la_n \}_{n \in \mathbb Z}$ be an arbitrary sequence of nonzero complex numbers in the form
$$
    \la_n = (n + a + \varkappa_n)^2, \quad n \in \mathbb Z, \quad a \in \mathbb C.
$$
Then the function 
$$
    P(\la) := \prod_{n = -\iy}^{\iy} \left( 1 - \frac{\la}{\la_n} \right)
$$
admits the representation
$$
    P(\la) = C (\cos 2 \rho \pi - \cos 2 a \pi) + \varkappa_{2, even}(\rho),
$$
where $C$ is a nonzero constant.
\end{lem}

\begin{thm} \label{thm:complete}
Under assumptions ($A_1$)--($A_4$), the system of vector functions $\mathcal S$, defined in \eqref{defS}, is complete in $\mathcal H$.
\end{thm}

\begin{proof}
Let the vector function $h  = \begin{bmatrix} h_1 \\ h_2 \end{bmatrix} \in \mathcal H$ be such that $(h(t), s(t, \la))_{\mathcal H} = 0$
for all $\la \in \Lambda'(L)$ and $(h(t), s_j(t))_{\mathcal H} = 0$ for $j = \overline{0, r-2}$, i.e.
\begin{gather} \label{sm2}
	\int_0^{l\pi} \left( h_1(t) \Delta_{known}^K(\la) \frac{\sin \rho t}{\rho} + h_2(t) \Delta_{known}^{\Pi}(\la) \cos \rho t \right)\, dt = 0,
	\quad \la \in \Lambda'(L), \\ \label{sm4}
\int_0^{l\pi} h_1(t) t^{2 j + 1} \, dt = 0, \quad \int_0^{l\pi} h_2(t) t^{2 j} \, dt = 0, \quad j = \overline{0, (r-3)/2}.
\end{gather}
Taking the relation $\Delta(\la) = 0$ for $\la \in \Lambda'(L)$ and assumption ($A_2$) into account, we derive from \eqref{DeltaPiK}, that
\begin{equation} \label{sm3}
    \Delta_{known}^K(\la) = -\Delta_{known}^{\Pi}(\la) \frac{\Delta_{unknown}^K(\la)}{\Delta_{unknown}^{\Pi}(\la)}, \quad \la \in \Lambda'(L).
\end{equation}
Substituting \eqref{sm3} into \eqref{sm2} and using ($A_2$) again, we obtain
$$
\int_0^{l\pi} \left( h_1(t) \Delta_{unknown}^K(\la) \frac{\sin \rho t}{\rho} - h_2(t) \Delta_{unknown}^{\Pi}(\la) \cos \rho t \right)\, dt = 0, \quad \la \in \Lambda'(L).
$$
Consequently, the entire function
\begin{equation} \label{defH}
   H(\la) := \rho^{r-1} \int_0^{l\pi} \left( h_1(t) \Delta_{unknown}^K(\la) \frac{\sin \rho t}{\rho} - h_2(t) \Delta_{unknown}^{\Pi}(\la) \cos \rho t \right)\, dt
\end{equation}
has zeros at the points $\la \in \Lambda'(L)$. Furthermore, relation \eqref{sm4} implies
\begin{equation} \label{hzero}
\int_0^{l \pi} h_1(t) \frac{d^j}{d \la^j} \left( \frac{\sin \rho t}{\rho} \right)\Big|_{\la = 0} \, dt = 0, \quad
\int_0^{l \pi} h_2(t) \frac{d^j}{d \la^j} \cos \rho t \Big|_{\la = 0} \, dt = 0, \quad j = \overline{0, (r-3)/2}.
\end{equation}
These relations together with \eqref{defH} and ($A_3$) yield that the function $H(\la)$ has the zero $\la = 0$ of multiplicity $r-1$.

Using \eqref{DK}, \eqref{DPi} and \eqref{defH}, we obtain the estimate
\begin{equation} \label{estH}
H(\la) = O(|\rho|^{-1}\exp(2 |\mbox{Im}\,\rho| l \pi)).
\end{equation}

Consider the function
$$
D(\la) := \la^{r-1}  \prod_{n \in \mathbb Z \backslash \overline{1, r-1}} \left( 1 - \frac{\la}{\la_{n1}} \right) \prod_{k = 2}^l \prod_{n = -\iy}^{\iy} \left( 1 - \frac{\la}{\la_{nk}} \right).
$$
In view of ($A_1$), the function $\dfrac{H(\la)}{D(\la)}$ is entire. Using asymptotic relations \eqref{asymptla3} and Lemma~\ref{lem:prod}, we obtain the estimate
$$
|D(\rho^2)| \ge C \exp(2 |\mbox{Im}\,\rho| l \pi), \quad \eps \le \arg \rho \le \pi - \eps, \quad |\rho| \ge \rho^*,
$$
for some positive $\eps$ and $\rho^*$. Applying this estimate together with \eqref{estH}, we get
$$
	\frac{H(\la)}{D(\la)} = O(|\rho|^{-1}), \quad \la = \rho^2, \quad \eps \le \arg \rho \le \pi - \eps, \: |\rho| \to \iy. 
$$
Using Phragmen-Lindel\"of's and Liouville's theorems, we conclude that $H(\la) \equiv 0$.

Denote by $\{ \mu_n \}_{n \in \mathbb N}$ the zeros of the characteristic function $\Delta_{unknown}^K(\la)$.
For simplicity, let all these zeros be simple. The case of multiple zeros requires minor modifications. Assumption ($A_3$) implies $\mu_n \ne 0$.
By virtue of ($A_4$), $\Delta_{unknown}^{\Pi}(\mu_n) \ne 0$. Therefore it follows from \eqref{defH} and the relation $H(\la) \equiv 0$, that the entire function
$$
     H_2(\la) := \int_0^{l\pi} h_2(t) \cos \rho t \, dt
$$
has zeros at $\la = \mu_n$, $n \in \mathbb N$. Moreover, in view of \eqref{hzero}, the function $H_2(\la)$ has the zero $\la = 0$ of multiplicity $\frac{r-1}{2}$. Consequently, the function
$\dfrac{H_2(\la)}{\rho^{r-1} \Delta_{unknown}^K(\la)}$ is entire. Using \eqref{DK}, one can obtain the following estimate
$$
|\rho^{r-1} \Delta_{unknown}^K(\rho^2)| \ge C \exp(|\mbox{Im}\, \rho| l \pi), \quad \eps \le \arg \rho \le \pi - \eps, \quad |\rho| \ge \rho^*,
$$
for some positive $\eps$ and $\rho^*$. On the other hand, $H_2(\la) = \varkappa_{l, even}(\rho)$. Applying Phragmen-Lindel\"of's and Liouville's theorems, we conclude that $H_2(\la) \equiv 0$ and, consequently, $h_2 = 0$ in $L_2(0, l\pi)$. Using \eqref{defH} again, we get that $h_1 = 0$. Thus, $h = 0$ and the system $\mathcal S$ is complete in~$\mathcal H$.
\end{proof}

\begin{remark} \label{rem:add}
If the potentials are unknown on a sufficiently large part of the graph $G$, the spectrum $\Lambda(L)$ can contain insufficient number of sequences in the form \eqref{asymptla3} to construct a complete system $\mathcal S$. Then one can consider boundary value problems of the same form as $L$ with changed boundary conditions in some vertices from $\partial G_{known}$, and use eigenvalues of those problems in the set $\Lambda'(L)$ (see Example~\ref{ex:2}).
\end{remark}

\begin{remark} \label{rem:A3}
For $r = 1$ assumption ($A_3$) is not necessary.
\end{remark}

Note that the proof of Theorem~\ref{thm:complete} essentially uses assumptions ($A_1$)--($A_4$). Although assumptions ($A_1$) and ($A_3$) are just technical, ($A_2$) and ($A_4$) are principal. One can easily provide examples of star-shaped quantum graphs with zero potentials, when these conditions are violated. Therefore the question arises, if there exist quantum graphs, satisfying assumptions ($A_1$)-($A_4$). Further we show, how such examples can be constructed for a tree with an arbitrary geometrical structure and arbitrary boundary conditions.

First, we provide a simple sufficient condition for ($A_2$) and ($A_4$). Denote by $G_j$, $j = \overline{1, n}$, $n := |E_w|$, the subtrees, obtained from $G$ by splitting the vertex $w$. Consider the characteristic functions $\Delta_j^D(\la)$ and $\Delta_j^N(\la)$ corresponding to $G_j$ with the Dirichlet and the Neumann boundary condition in $w$, respectively, and the conditions \eqref{MC} and BC in the other vertices. Introduce
$$
\Delta^{\Pi}(\la) := \prod_{j = 1}^n \Delta_j^D(\la).
$$

One can easily show that the following condition implies assumptions ($A_2$) and ($A_4$):

\smallskip

($A_5$) The characteristic functions $\Delta(\la)$ and $\Delta^{\Pi}(\la)$ do not have common zeros.

\smallskip

Condition ($A_5$) holds, if no two functions among $\Delta_j^D(\la)$ have common zeros and $\Delta_j^D(\la)$ does not have common zeros with $\Delta_j^N(\la)$ for $j = \overline{1, n}$. Suppose that, on the contrary, $\Delta_j^D(\la_0) = 0$ and $\Delta_j^N(\la_0) = 0$ for some $\la_0$. For simplicity, denote by $e_j$ the edge, incident to $w$ in the graph $G_j$, and by $G_{j, next}$ the tree, obtained from $G_j$ by removing the vertex $w$ and the edge $e_j$. The notations $\Delta_{j, next}^K(\la)$ and $\Delta_{j, next}^{\Pi}(\la)$ will be used for the characteristic functions for the graph $G_{j, next}$ similar to $\Delta_{unknown}^K(\la)$ and $\Delta_{unknown}^{\Pi}(\la)$ for $G_{unknown}$. The second vertex, incident to $e_j$, plays the role of $w$. Then the following relations hold
\begin{align*}
\Delta_j^D(\la_0) & = S_j(\pi, \la_0) \Delta_{j, next}^K(\la_0) + S_j^{[1]}(\pi, \la_0) \Delta_{j, next}^{\Pi}(\la_0) = 0, \\
\Delta_j^N(\la_0) & = C_j(\pi, \la_0) \Delta_{j, next}^K(\la_0) + C_j^{[1]}(\pi, \la_0) \Delta_{j, next}^{\Pi}(\la_0) = 0.
\end{align*}
Using equation \eqref{eqv} and the initial conditions \eqref{ic}, one can show that the determinant of this system equals
$$
S_j(\pi, \la_0) C_j^{[1]}(\pi, \la_0) - S_j^{[1]}(\pi, \la_0) C_j(\pi, \la_0) = 1.
$$
Consequently, we conclude that $\la_0$ is a common zero of $\Delta_{j, next}^K(\la)$ and $\Delta_{j, next}^{\Pi}(\la)$. 

Thus, we have proved that, if condition ($A_5$) holds for all the subtrees $G_{j, next}$, $j = \overline{1, n}$, and the characteristic functions $\Delta_j^D(\la)$ pairwise do not have common zeros, then ($A_5$) holds for the whole graph $G$. We use this fact to construct a boundary value problem on the tree $G$, satisfying assumption ($A_5$), recursively.

Consider an arbitrary boundary value problem $L$ on a tree $G$. Suppose that assumption ($A_5$) does not hold for $G$. If $G$ is a star-shaped graph, this means that some pairs of the functions $\Delta_j^D(\la)$ have common zeros. For each fixed $j = \overline{1, n}$, the zeros of $\Delta_j^D(\la)$ are the eigenvalues of the corresponding Sturm-Liouville boundary value problem on the finite interval. By the transfrom $\sigma_j \to \sigma_j + c_j x_j$, $\ga_j \to \ga_j - c_j \pi$, we shift the corresponding spectrum by a constant $c_j$. 
One can always choose constants $c_j$, $j = \overline{1, n}$ to make the eigenvalues for all the edges distinct.

If the tree $G$ has a more complicated structure, we suppose that we already have applied our algorithm recursively and achieved condition ($A_5$) for the subtrees $G_{j, next}$, $j = \overline{1, n}$. Consequently, the characteristic functions $\Delta_j^D(\la)$ and $\Delta_j^N(\la)$ can not have common zeros for every fixed $j = \overline{1, n}$. Then for each $j = \overline{1, n}$, we can add a constant $c_j$ to all the potentials on the tree $G_j$ and make the zeros of the different functions $\Delta_j^D(\la)$ distinct. Note that the shift does not affect condition ($A_5$) for subtrees. After such shifts, we can add a large positive constant to all the potentials to make all the eigenvalues of $L$ positive and to achieve condition ($A_3$). Thus, we can change potentials for any initial graph to satisfy assumptions ($A_1$)--($A_4$), probably, except for the requirement of distinct eigenvalues in ($A_1$), which is not so important.

\bigskip

{\large \bf 5. Examples}

\bigskip

In this section, we consider two examples, illustrating our method for solution of the partial inverse problems on trees. The results of works \cite{Bond17-1, Bond17-2, Bond17-preprint} for the star-shaped graphs with edges of equal length can also be considered as examples of the technique developed in this paper. 

\begin{example} \label{ex:1}
Consider the graph in Figure~\ref{img:2}. It has five edges $\{ e_j \}_{j = 1}^5$. The length of the edge $e_j$ equals $T_j$, where $T_j = \pi$ for $j = 1, 2, 3$ and $T_j = 2 \pi$ for $j = 4, 5$.
For brevity, we shall not divide the edges $e_4$ and $e_5$ into smaller ones. Consider the system of the Sturm-Liouville equations with singular potentials:
$$
-(y_j^{[1]})' - \sigma_j(x_j) y_j^{[1]} - \sigma_j^2(x_j) y_j = \la y_j, \quad x_j \in (0, T_j), \quad j = \overline{1, 5}.
$$

\begin{figure}[h!]
\centering
\begin{tikzpicture}
\filldraw (0, 0) circle (1pt) node[anchor=south]{$w$};
\filldraw (2, 0) circle (1pt) node[anchor=north]{$\pi$};
\draw (0, 0) edge node[auto]{$e_1$} (2, 0);
\draw (0.2, -0.2) node{$0$};
\filldraw (3.5, 0.7) circle (1pt) node[anchor=north]{$0$};
\draw (2, 0) edge node[auto]{$e_2$} (3.5, 0.7);
\filldraw (3.5, -0.7) circle (1pt) node[anchor=south]{$0$};
\draw (2, 0) edge node[anchor=north]{$e_3$} (3.5, -0.7);  
\draw[dashed] (0, 0) edge [bend left] (3.7, 1);
\draw[dashed] (0, 0) edge [bend right] (3.7, -1);
\draw[dashed] (3.7, -1) arc(-80:80:1);
\draw (2.5, -2) node{$G_{unknown}$};
\draw (-0.6, 0) node{$2 \pi$};
\filldraw (-3, 1.2) circle (1pt) node[anchor=north]{$0$};
\draw (0, 0) edge node[anchor=south]{$e_4$} (-3, 1.2);
\filldraw (-3, -1.2) circle (1pt) node[anchor=south]{$0$};
\draw (0, 0) edge node[auto]{$e_5$} (-3, -1.2);
\draw[dashed] (0, 0) edge [bend right] (-2.7, 1.5);
\draw[dashed] (0, 0) edge [bend left] (-2.7, -1.5);
\draw[dashed] (-2.7, 1.5) arc(90:270:1.5);
\draw (-2, -2) node{$G_{known}$};

\end{tikzpicture}
\caption{Example~\ref{ex:1}}
\label{img:2}
\end{figure}

The standard matching conditions \eqref{MC} take the form
\begin{gather*}
y_1(\pi) = y_2(\pi) = y_3(\pi), \quad y_1^{[1]}(\pi) + y_2^{[1]}(\pi) + y_3^{[1]}(\pi) = 0, \\
y_1(0) = y_4(2 \pi) = y_5(2 \pi), \quad -y_1^{[1]}(0) + y_4^{[1]}(2 \pi) + y_5^{[1]}(2 \pi) = 0.
\end{gather*}
Impose the following boundary conditions BC:
$$
y_2^{[1]}(0) = 0, \quad y_3(0) = 0, \quad y_4^{[1]}(0) = 0, \quad y_5(0) = 0.
$$

The edges $e_1$, $e_2$, $e_3$ belong to $G_{unknown}$, and the edges $e_4$, $e_5$ belong to $G_{known}$. Using \eqref{Dfact}, we obtain the following representations for the characteristic functions:
\begin{equation*} 
\arraycolsep=1.4pt\def\arraystretch{1.5}
\begin{array}{l}    
\Delta_{known}^K(\la) = C_4^{[1]}(2\pi, \la) S_5(2\pi, \la) + C_4(2\pi, \la) S_5^{[1]}(2\pi, \la), \\
\Delta_{known}^{\Pi}(\la) = C_4(2\pi, \la) S_5(2\pi, \la), \\
\Delta_{unknown}^K(\la) = C_1^{[1]}(\pi, \la) C_2(\pi, \la) S_3(\pi, \la) + C_1(\pi, \la) C_2^{[1]}(\pi, \la) S_3(\pi, \la) + C_1(\pi, \la) C_2(\pi, \la) S_3^{[1]}(\pi, \la), \\ 
\Delta_{unknown}^{\Pi}(\la) = S_1^{[1]}(\pi, \la) C_2(\pi, \la) S_3(\pi, \la) + S_1(\pi, \la) C_2^{[1]}(\pi, \la) S_3(\pi, \la) + S_1(\pi, \la) C_2(\pi, \la) S_3^{[1]}(\pi, \la).
\end{array}
\end{equation*}
Substituting \eqref{CS} into these relations, we calculate
\begin{equation} \label{Dex1}
\arraycolsep=1.4pt\def\arraystretch{1.5}
\left.
\begin{array}{l}    
\Delta_{known}^K(\la) = \cos 4 \rho \pi + \varkappa_{4, even}(\rho), \\
\Delta_{known}^{\Pi}(\la) = \dfrac{\sin 4 \rho \pi}{2 \rho} + \dfrac{\varkappa_{4, odd}(\rho)}{\rho}, \\
\Delta_{unknown}^K(\la) = 3 \cos^3 \rho \pi - 2 \cos \rho \pi  + \varkappa_{3, even}(\rho), \\ 
\Delta_{unknown}^{\Pi}(\la) = \dfrac{3 \sin \rho \pi \cos^2 \rho \pi}{\rho} - \dfrac{\sin \rho \pi}{\rho} + \dfrac{\varkappa_{3, odd}(\rho)}{\rho}.
\end{array}\right\}
\end{equation}
Substituting \eqref{Dex1} into \eqref{DeltaPiK}, we get
$$
\Delta(\la) = \frac{\sin \rho \pi}{\rho} \left( 36 \cos^6 \rho \pi - 46 \cos^4 \rho \pi + 15 \cos^2 \rho \pi - 1\right) + \frac{\varkappa_{7, odd}(\rho)}{\rho}.
$$

Consequently, the spectrum $\Lambda(L)$ contains three subsequences in the form \eqref{asymptla3} with
$$
\al_1 \approx 0.20433, \quad \al_2 \approx 0.37334, \quad \al_3 \approx 0.47147.
$$
Since $r = 1$, we put $\Lambda'(L) := \{ \la_{nk} \}_{n \in \mathbb Z, \, k = 1, 2, 3}$. In this case,
$\mathcal S = \{ s(t, \la_{nk}) \}_{n \in \mathbb Z, \, k = 1, 2, 3}$. If assumptions ($A_1$)-($A_4$) hold, by virtue of Theorem~\ref{thm:complete}, the system $\mathcal S$ is complete in $\mathcal H$.
In view of Theorem~\ref{thm:uniqap}, the characteristic functions $\Delta_{unknown}^K(\la)$ and $\Delta_{unknown}^{\Pi}(\la)$ are uniquely specified by $\Lambda'(L)$, given $\sigma_4$ and $\sigma_5$.

According to Remark~\ref{rem:1edge}, we recover $\sigma_1$ on the only edge $e_1$ incident to $w$ in the subtree $G_{unknown}$. Let now the graph $G_{known}$ contain the edges $e_1$, $e_4$, $e_5$, and the graph $G_{unknown}$ contain the edges $e_2$ and $e_3$. Then the characteristic functions have the following form
\begin{equation*} 
\arraycolsep=1.4pt\def\arraystretch{1.5}
\begin{array}{l}    
\Delta_{known}^K(\la) = 12 \cos^5 \rho \pi - 14 \cos^3 \rho \pi + 3 \cos \rho \pi + \varkappa_{5, even}(\rho), \\
\Delta_{known}^{\Pi}(\la) = \dfrac{\sin \rho \pi}{\rho} \left(12 \cos^4 \rho \pi - 10\cos^2 \rho \pi + 1 \right) + \dfrac{\varkappa_{5, odd}(\rho)}{\rho}, \\
\Delta_{unknown}^K(\la) = \cos 2 \rho \pi  + \varkappa_{2, even}(\rho), \\ 
\Delta_{unknown}^{\Pi}(\la) = \dfrac{\sin 2 \rho \pi}{2 \rho} + \dfrac{\varkappa_{2, odd}(\rho)}{\rho}.
\end{array}
\end{equation*}

Suppose that the new $\Delta_{unknown}^K(\la)$ and $\Delta_{unknown}^{\Pi}(\la)$ can be constructed by some subspectrum of the problem $L$. Consider the boundary value problem $L_1$ of the same form as $L$ with the Dirichlet boundary condition $y_2(0) = 0$ instead of the Neumann one. We have
\begin{equation*} 
\arraycolsep=1.4pt\def\arraystretch{2}
\begin{array}{l}    
\Delta_{unknown, 1}^K(\la) = \dfrac{\sin 2 \rho \pi}{\rho} + \dfrac{\varkappa_{2, odd}(\rho)}{\rho}, \\
\Delta_{unknown, 1}^{\Pi}(\la) = \dfrac{\sin^2 \rho \pi}{\rho^2} + \dfrac{\varkappa_{2, even}(\rho)}{\rho^2}, \\
\Delta_1(\la) = \dfrac{\sin^2 \rho \pi \cos \rho \pi}{\rho^2} \left( 36 \cos^4 \rho \pi - 34 \cos^2 \rho \pi + 5\right) + \dfrac{\varkappa_{7, even}(\rho)}{\rho^2}.
\end{array}
\end{equation*}
The functions $\Delta_{known}^K(\la)$ and $\Delta_{known}^{\Pi}(\la)$ remain the same as for $L$.

One can easily show that $\Delta_1(\la)$ has two sequences of zeros in the form \eqref{asymptla3} with
$$
\al_1 \approx 0.22410, \quad \al_2 \approx 0.44167.
$$
Using these subsequences, under some additional conditions, one can construct a complete system $\mathcal S_1$ and find the functions $\Delta_{unknown, 1}^K(\la)$ and $\Delta_{unknown, 1}^{\Pi}(\la)$. Note that $\Delta_{unknown, 1}^K(\la)$ and $\Delta_{unknown}^K(\la)$ are the characteristic functions for the Sturm-Liouville boundary value problems on an interval of length $2 \pi$, consisting of the edges $e_2$ and $e_3$, with different boundary conditions. One can recover $\sigma_2$ and $\sigma_3$, solving the inverse problem by two spectra on the interval \cite{HM04-2spectra}.

Reconstruction of $\sigma_2$ and $\sigma_3$ can be simplified in the way described in \cite{Bond17-preprint}. Indeed, note that
$$
\frac{\Delta_{unknown}^K(\la)}{\Delta_{unknown}^{\Pi}(\la)} = \frac{C_2^{[1]}(\pi, \la)}{C_2(\pi, \la)} + \frac{S_3^{[1]}(\pi, \la)}{S_3(\pi, \la)}, \qquad
\frac{\Delta_{unknown, 1}^K(\la)}{\Delta_{unknown, 1}^{\Pi}(\la)} = \frac{S_2^{[1]}(\pi, \la)}{S_2(\pi, \la)} + \frac{S_3^{[1]}(\pi, \la)}{S_3(\pi, \la)}.
$$
The fraction $\dfrac{\Delta_{unknown}^K(\la)}{\Delta_{unknown}^{\Pi}(\la)}$ is already known. For the values $\la \in \Lambda(L_1)$, such that $\Delta_{known}^{\Pi}(\la) \ne 0$ and
$\Delta_{unknown, 1}^{\Pi}(\la) \ne 0$, one can find the fraction $\dfrac{\Delta_{unknown, 1}^K(\la)}{\Delta_{unknown, 1}^{\Pi}(\la)}$ from the relation similar to \eqref{sm3}. Then for such values of $\la$, we know the expression
$$
\frac{S_2^{[1]}(\pi, \la)}{S_2(\pi, \la)} - \frac{C_2^{[1]}(\pi, \la)}{C_2(\pi, \la)} = \frac{1}{S_2(\pi, \la) C_2(\pi, \la)}.
$$

One can interpolate the function $S_2(\pi, \la) C_2(\pi, \la)$ by one subsequence in the form \eqref{asymptla3}. The zeros of the functions $S_2(\pi, \la)$ and $C_2(\pi, \la)$ interlace (see \cite{HM04-2spectra}), so we can find them separately and recover the potential on $e_2$ from two spectra. Subsequently, it is easy to find $\sigma_3$. 

This example involves complicated calculations, so we demonstrated only the general scheme and did not investigate the questions of the Riesz-basicity. However, we shall study these questions for the next example.
\end{example}

\begin{example} \label{ex:2}
Consider the graph in Figure~\ref{img:3}. It has three edges: $e_1$ of length $2 \pi$ and $e_2$, $e_3$ of length $\pi$. Consider the boundary
value problem $L$ for the Sturm-Liouville equations on this graph with the standard matching conditions
\begin{equation} \label{MCex2}
    y_1(2 \pi) = y_2(\pi) = y_3(\pi), \quad y_1^{[1]}(2\pi) + y_2^{[1]}(\pi) + y_3^{[1]}(\pi) = 0
\end{equation}
and the boundary conditions
$$
  	y_1(0) = 0, \quad y_2^{[1]}(0) = 0, \quad y_3^{[1]}(0) = 0.
$$

\begin{figure}[h!]
\centering
\begin{tikzpicture}
\filldraw (0, 0) circle (1pt) node[anchor=south]{$w$};
\filldraw (4, 0) circle (1pt) node[anchor=north]{$0$};
\draw (0.3, -0.2) node{$2 \pi$};
\draw (0, 0) edge node[auto]{$e_1$} (4, 0);
\draw[dashed] (0, 0) edge [bend left] (4, 0.5);
\draw[dashed] (0, 0) edge [bend right] (4, -0.5);
\draw[dashed] (4, -0.5) arc(-80:80:0.5);
\draw (2.5, -1.3) node{$G_{unknown}$};
\draw (-0.5, 0) node{$\pi$};
\filldraw (-1.8, 0.8) circle (1pt) node[anchor=north]{$0$};
\draw (0, 0) edge node[anchor=south]{$e_2$} (-1.8, 0.8);
\filldraw (-1.8, -0.8) circle (1pt) node[anchor=south]{$0$};
\draw (0, 0) edge node[anchor=north]{$e_3$} (-1.8, -0.8);
\draw[dashed] (0, 0) edge [bend right] (-1.5, 1);
\draw[dashed] (0, 0) edge [bend left] (-1.5, -1);
\draw[dashed] (-1.5, 1) arc(90:270:1);
\draw (-1.2, -1.3) node{$G_{known}$};

\end{tikzpicture}
\caption{Example~\ref{ex:2}}
\label{img:3}
\end{figure}

Let $G_{unknown}$ consist of the only edge $e_1$, and let $G_{known}$ consist of the edges $e_2$, $e_3$. The characteristic function
\begin{multline*}
  \Delta(\la) = S_1^{[1]}(2\pi, \la) C_2(\pi, \la) C_3(\pi, \la) + S_1(2\pi, \la) C_2^{[1]}(\pi, \la) C_3(\pi, \la) \\ + 
  S_1(2\pi, \la) C_2(\pi, \la) C_3^{[1]}(\pi, \la)  = \cos^2 \rho \pi (6 \cos^2 \rho \pi - 5) + \varkappa_{4, even}(\rho)
\end{multline*}
has the only subsequence of zeros in the form \eqref{asymptla3} with $\al_1 = \frac{1}{\pi} \arccos \sqrt{\frac{5}{6}}$. In order
to recover the potential $\sigma_1$ on the edge of length $2\pi$, we need two such subsequences. Following Remark~\ref{rem:add}, consider 
another boundary value problem $L_1$ with the matching conditions \eqref{MCex2} and the boundary conditions
$$
   y_1(0) = 0, \quad y_2(0) = 0, \quad y_3^{[1]}(0) = 0.
$$
Its characteristic function admits the representation
$$
 \Delta_1(\la) = \frac{3 \sin \rho \pi \cos \rho \pi}{\rho} (2 \cos^2 \rho \pi - 1) + \frac{\varkappa_{4, odd}(\rho)}{\rho},
$$
and also has the only subsequence of zeros in the form \eqref{asymptla3} with $\al_2 = \frac{1}{4}$. 

Define $\Lambda' := \{ \la_{nk} \}_{n \in \mathbb Z, k = 1, 2}$, where $\{ \la_{n1} \}_{n \in \mathbb Z}$ 
and $\{ \la_{n2} \}_{n \in \mathbb Z}$ are the mentioned subsequences for the problems $L$ and $L_1$, respectively.
Introduce the vector functions
\begin{equation} \label{defsex2}
  	s(t, \la_{n1}) = \begin{bmatrix}
  	                     \Delta_{known}^K(\la_{n1}) \rho_{n1}^{-1} \sin \rho_{n1} t \\
  	                     \Delta_{known}^{\Pi}(\la_{n1}) \cos \rho_{n1} t 
  					 \end{bmatrix},
  	\quad s(t, \la_{n2}) = \begin{bmatrix}
  							\Delta_{known, 1}^K(\la_{n2}) \rho_{n2}^{-1}\sin \rho_{n2} t \\
  	                        \Delta_{known, 1}^{\Pi}(\la_{n2}) \cos \rho_{n2} t 
  						   \end{bmatrix}	 
  	\quad n \in \mathbb Z,
\end{equation}
where the characteristic functions $\Delta_{known, 1}^K(\la)$ and $\Delta_{known, 1}^{\Pi}(\la)$ correspond to the boundary value problem $L_1$.

Suppose that assumption ($A_1$) holds for the set $\Lambda'$, and assumption ($A_2$) holds in the following modified form: $\Delta_{known}^{\Pi}(\la_{n1}) \ne 0$, $\Delta_{known, 1}^{\Pi}(\la_{n2}) \ne 0$, $\Delta_{unknown}^{\Pi}(\la_{nk}) \ne 0$ for all $n \in \mathbb Z$, $k = 1, 2$. Assumption ($A_3$) is not required in view of Remark~\ref{rem:A3}.
Assumption ($A_4$) holds automatically, since
the functions $\Delta_{unknown}^K(\la) = S_1^{[1]}(2\pi, \la)$ and $\Delta_{unknown}^{\Pi}(\la) = S_1(2\pi, \la)$ do not have common zeros
as characteristic functions of Sturm-Liouville problems on an interval with boundary conditions, different at one of the end points.

Modifying the proof of Theorem~\ref{thm:complete}, one can show that the system \eqref{defsex2} is complete in $\mathcal H$. 
Consequently, by Theorem~\ref{thm:uniqap} functions $\Delta_{unknown}^K(\la)$ and $\Delta_{unknown}^{\Pi}(\la)$
are uniquely specified by the set $\Lambda'$, if $\sigma_2$ and $\sigma_3$ are known a priori.  
Thus, we arrive at the inverse problem on the interval, which consists in recovering $\sigma_1$
from the two characteristic functions $S_1(\pi, \la)$ and $S_1^{[1]}(\pi, \la)$ (see Remark~\ref{rem:1edge}). 
This problem is uniquely solvable (see \cite{HM04-2spectra}).

Let us prove that the system \eqref{defsex2} is an unconditional basis. In view of relation \eqref{sm3} and assumptions~($A_1$)--($A_2$), 
we can study the system of the vector functions
$$
    s_{nk}(t) = \begin{bmatrix}
                     \Delta_{unknown}^K(\la_{nk}) \sin \rho_{nk} t \\
                     -\rho_{nk} \Delta_{unknown}^{\Pi}(\la_{nk}) \cos \rho_{nk} t 
   				  \end{bmatrix}
   \quad n \in \mathbb Z, \: k = 1, 2, 
$$
instead of \eqref{defsex2}. The system $\{ s_{nk} \}_{n \in \mathbb Z, \, k = 1, 2}$ is also complete. Using asymptotic formulas \eqref{asymptla3}, 
one can easily show that it is $l_2$-close to the system of the vector functions
$$
    s_{nk}^0(t) = \begin{bmatrix}
    				\cos 2 \al_k \pi \sin (n + \al_k) t \\
    				-\sin 2 \al_k \pi \cos (n + \al_k) t
    		     \end{bmatrix},
   \quad n \in \mathbb Z, \: k = 1, 2, 		      	 
$$
i.e. $\{ \| s_{nk} - s_{nk}^0 \|_{\mathcal H} \} \in l_2$. Note that $\cos 2 \al_2 \pi = 0$ and $\cos 2 \al_1 \pi \ne 0$.
Since the systems $\{ \sin (n + \al_1) t \}_{n \in \mathbb Z}$ and $\{ \cos (n + \al_2) t \}_{n \in \mathbb Z}$ are Riesz bases
in $L_2(0, 2\pi)$ (see \cite{Bond17-preprint}), one can easily show that $\{ s_{nk}^0 \}_{n \in \mathbb Z, \, k = 1, 2}$ is a Riesz
basis in $\mathcal H$. Thus, the system $\{ s_{nk} \}_{n \in \mathbb N, \, k = 1, 2}$ is complete and $l_2$-close to the Riesz basis, 
so it is also a Riesz basis. Hence the normalized system \eqref{defsex2} is a Riesz basis in $\mathcal H$, and one can apply 
Algorithm~\ref{alg:ap}, and subsequently solve the partial inverse problem by the subspectrum $\Lambda'$.
\end{example} 

{\bf Acknowledgment.} 
This work was supported by the Russian Federation
President Grant MK-686.2017.1, by Grant 1.1660.2017/4.6 of the Russian
Ministry of Education and Science, and by Grants 15-01-04864,
16-01-00015, 17-51-53180 of the Russian Foundation for Basic Research.

\medskip

\noindent Natalia Pavlovna Bondarenko \\
1. Department of Applied Mathematics, Samara National Research University, \\
Moskovskoye Shosse 34, Samara 443086, Russia, \\
2. Department of Mechanics and Mathematics, Saratov State University, \\
Astrakhanskaya 83, Saratov 410012, Russia, \\
e-mail: {\it BondarenkoNP@info.sgu.ru}


\begin{thebibliography}{99}

\bibitem{RS53}
Ruedenberg, K.; Scherr, C. W. Free-electron network model for conjugated systems. I. Theory, J. Chem. Physics,
21:9 (1953), 1565--1581.

\bibitem{Griff53}
Griffith, J. S. A free-electron theory of conjugated molecules. I. Polycyclic Hydrocarbons, Trans. Faraday Soc., 49 (1953), 
345--351.

\bibitem{Dat99}
Datta, S. Electronic Transport in Mesoscopic Systems. Cambridge: Cambridge Univ. Press, 1999.

\bibitem{Imry97}
Imry, Y. Introduction to Mesoscopic Physics (Mesoscopic Physics and Nanotechnology), Oxford Univ. Press, 1997.

\bibitem{Mur01}
Murayama, Y. Mesoscopic systems, John Wiley \& Sons, 2001.	

\bibitem{KK02}
Kuchment, P.; Kunyansky, L. Differential operators on graphs and photonic crystals. Adv. Comput. Math. 16 (2002), 263--290.

\bibitem{ES89}
Exner, P. \v Seba, P. Electrons in semiconductor microstructures: a challenge to operator theorists, 
Schr\"odinger Operators, Standard and Nonstandard (Dubna 1988), World Scientific, Singapore 1989, 79--100.

\bibitem{CLMT97}
Carini, J. P.; Londergan, J. T.; Murdock, D. P.; Trinkle, D.; Yung, C. S. Bound states in waveguides and bent quantum wires.
I. Applications to waveguide systems, II. Electrons in quantum wires, Phys. Rev. B 55 (1997), 9842--9859.

\bibitem{KS97}
Kottos, T.; Smilansky, U. Quantum chaos on graphs, Phys. Rev. Lett. 79 (1997), 4794--4797. 

\bibitem{BCFK06}
Berkolaiko, G.; Carlson, R.; Fulling, S.; Kuchment, P. Quantum Graphs and Their Applications,
Contemp. Math. 415, Amer. Math. Soc., Providence, RI, 2006.

\bibitem{BK13}
Berkolaiko, G.; Kuchment, P. Introduction to Quantum Graphs, Amer. Math. Soc., Providence,
RI, 2013.

\bibitem{GS06}
Gnutzmann, S.; Smilansky, U. Quantum graphs: applications to quantum chaos and universal
spectral statistics, Adv. Phys. 55 (2006), 527--625.

\bibitem{Post12}
Post, O. Spectral Analysis on Graph-like Spaces, Lecture Notes in Mathematics, vol. 2039. Springer
Verlag, Berlin, 2012.


\bibitem{PPP04}
Pokorny, Yu. V.; Penkin, O. M.; Pryadiev, V. L. et al. Differential Equations on Geometrical Graphs,
Fizmatlit, Moscow, 2004 (Russian).

\bibitem{Exner08}
Analysis on Graphs and Its Applications, edited by P. Exner, J. P. Keating, P. Kuchment,
T. Sunada and A. Teplyaev. Proceedings of Symposia in Pure Mathematics, AMS, 77.
(2008).

\bibitem{Kuch04}
Kuchment, P. Quantum graphs. Some basic structures. Waves Random Media 14 (2004),
S107--S128.

\bibitem{Ber16}
Berkolaiko, G. An elementary introduction to quantum graphs, preprint (2016), arXiv:1603.07356 [math-ph].

\bibitem{Yur16}
Yurko, V. A. Inverse spectral problems for differential operators on spatial networks, Russian Mathematical Surveys 71:3 (2016), 539--584.

\bibitem{Mar77}
Marchenko, V. A. Sturm-Liouville Operators and their Applications, Naukova Dumka,
Kiev (1977) (Russian); English transl., Birkhauser (1986).

\bibitem{Lev84}
Levitan, B.M. Inverse Sturm-Liouville Problems, Nauka, Moscow (1984) (Russian); English
transl., VNU Sci. Press, Utrecht (1987).

\bibitem{PT87}
P\"{o}schel, J; Trubowitz, E. Inverse Spectral Theory, New York, Academic Press (1987).

\bibitem{FY01}
Freiling, G.; Yurko, V. Inverse Sturm-Liouville Problems and Their Applications. Huntington,
NY: Nova Science Publishers, 2001.

\bibitem{GS01}
Gutkin, B.; Smilansky, U. Can one hear the shape of a graph? J. Phys. A, 34:31 (2001), 6061--6068.

\bibitem{KS00}
Kostrykin, V.; Schrader, R. Kirchhoff's rule for quantum wires. II: The inverse problem with possible applications to quantum computers,
Fortschritte der Physik 48 (2000), 703--716. 

\bibitem{KN05}
Kurasov, P.; Nowaczyk, M. Inverse spectral problem for quantum graphs. J. Phys. A, 38:22 (2005), 4901--4915.

\bibitem{AKN10}
Avdonin, S.; Kurasov, P.; Nowaczyk, M. Inverse Problems for Quantum
Trees II. Recovering Matching Conditions for Star Graphs, Inverse Problems
and Imaging, 4/4 (2010), 579--598.

\bibitem{EK12}
Ershova, Yu.; Kiselev, A, V. Trace formulae for graph Laplacians with applications to recovering matching conditions,
Methods Funct. Anal. Topology 18:4 (2012), 343--359.

\bibitem{Yur05} 
Yurko, V. A. Inverse spectral problems for Sturm-Liouville operators on graphs, Inverse Prob.
21 (2005), 1075--1086.

\bibitem{Yur06}
Yurko, V. A. On recovering Sturm-Liouville operators on graphs, Math. Notes 79:4 (2006), 572--782.

\bibitem{Yur02}
Yurko V.A. Method of Spectral Mappings in the Inverse Problem Theory, Inverse and
Ill-posed Problems Series. VSP, Utrecht, 2002.

\bibitem{Yur10}
Yurko, V. A. Inverse spectral problems for differential operators on arbitrary compact
graphs. J. Inverse and Ill-Posed Probl. 18:3 (2010), 245--261.

\bibitem{Borg46}
Borg, G. Eine Umkehrung der Sturm-Liouvilleschen Eigenwertaufgabe, Acta Math. 78 (1946), 1--96.

\bibitem{HL78}
Hochstadt, H.; Lieberman, B. An inverse Sturm-Liouville problem with mixed given data, SIAM J. Appl. Math. 34 (1978), 676--680.

\bibitem{GS00}
Gesztesy, F.; Simon, B. Inverse spectral analysis with partial information on the potential, II. The case of discrete spectrum,
Trans. AMS 352:6 (2000), 2765--2787. 

\bibitem{Sakh01}
Sakhnovich, L. Half-inverse problems on the finite interval, Inverse Problems 17 (2001), 527--532.

\bibitem{MP10}
Martinyuk, O.; Pivovarchik, V. On the Hochstadt-Lieberman theorem, Inverse Problems 26 (2010), 035011 (6pp).

\bibitem{HM04-half}
Hryniv, R.O.; Mykytyuk, Ya. V.  Half-inverse spectral problems for Sturm-Liouville
operators with singular potentials, Inverse Problems 20 (2004), 1423--1444.

\bibitem{Piv00}
Pivovarchik, V. N. Inverse problem for the Sturm-Liouville equation on a simple graph,
SIAM J. Math. Anal. 32:4 (2000), 801--819.

\bibitem{Yang10}
Yang, C.-F. Inverse spectral problems for the Sturm-Liouville operator on a $d$-star graph, 
J. Math. Anal. Appl. 365 (2010), 742--749.

\bibitem{Yur09}
Yurko, V. A. Inverse nodal problems for the Sturm-Liouville differential operators on a star-type graph,
Siberian Math. J. 50:2 (2009), 373-–378.

\bibitem{Yang11}
Yang, C.-F.; Yang X.-P. Uniqueness theorems from partial information of the potential on a
graph, J. Inverse Ill-Posed Prob. 19 (2011), 631-–639.

\bibitem{Yang17}
Yang, C.-F. Inverse problems for the differential operator on a graph with cycles, J. Math. Anal. Appl. 
J. Math. Anal. Appl. 445:2 (2017), 1548--1562. 

\bibitem{YW17}
Yang, C.-F.; Wang, F. Inverse problems on graphs with loops, J. Inverse Ill-Posed Probl. 25:3 (2017), 373--380.

\bibitem{Bond17-1}
Bondarenko, N. P. A partial inverse problem for the Sturm-Liouville operator on a star-shaped graph, 
Anal. Math. Phys., published online 24 April 2017, DOI: 10.1007/s13324-017-0172-x.

\bibitem{Bond17-2}
Bondarenko, N. P. Partial inverse problems for the Sturm-Liouville operator on a star-shaped graph
with mixed boundary conditions, J. Inverse Ill-Posed Probl., 
published online 2017-03-16, DOI: 10.1515/jiip-2017-0001.

\bibitem{Bond17-preprint}
Bondarenko, N. P. A 2-edge partial inverse problem for the Sturm-Liouville operators with singular potentials on a star-shaped graph,
preprint (2017), arXiv:1702.08293 [math.SP].

\bibitem{Bond17-3}
Bondarenko, N. P. A partial inverse problem for the differential pencil on a star-shaped graph, Results Math., published online 
03 May 2017, DOI: 10.1007/s00025-017-0683-7. 

\bibitem{BS17}
Bondarenko N., Shieh C.-T. Partial inverse problems on trees, Proceedings of the Royal Society of Edinburg Section A: Mathematics,
published online 15 August 2017, DOI: 10.1017/S0308210516000482.

\bibitem{SS99}
Savchuk, A. M.; Shkalikov, A. A. Sturm-Liouville operators with singular potentials, Math. Notes 66:6 (1999), 741--753.

\bibitem{Sav01}
Savchuk, A. M. On the eigenvalues and eigenfunctions of the Sturm-Liouville operator with a singular potential, Math. Notes
69:1--2 (2001), 245--252.

\bibitem{HM03}
Hryniv, R. O.; Mykytyuk, Ya. V. Inverse spectral problems for Sturm-Liouville operators
with singular potentials, Inverse Problems 19 (2003), 665-684.

\bibitem{HM04-2spectra}
Hryniv, R.O.; Mykytyuk, Ya. V. Inverse spectral problems for Sturm-Liouville operators with singular potentials, II. Reconstruction by two spectra, in: V. Kadets, W. Zelazko (Eds.), Functional Analysis and Its Applications, in: North-Holland Math. Stud., vol. 197, North-Holland Publishing, Amsterdam (2004), 97–114.

\bibitem{HM04-transform}
Hryniv, R. O.; Mykytyuk, Ya. V. Transformation operators for Sturm-Liouville operators with singular
potentials, Math. Phys. Anal. Geom. 7 (2004), 119--149.

\bibitem{FIY08}
Freiling, G.; Ignatiev, M.; Yurko, V. An inverse spectral problem for Sturm-Liouville operators with singular potentials on star-type graph, Proc. Symp. Pure Math. 77 (2008), 397--408.

\bibitem{LP09}
Law, C.-K.; Pivovarchik, V. Characteristic functions on quantum graphs, J. Phys. A: Math. Theor. 42 (2009), 035302.

\bibitem{LY12}
Law, C.-K.; Yanagida, E. A solution to an Ambarzumyan problem on trees, Kodai Math. J. 35:2 (2012), 358--373.

\bibitem{Yur09}
Yurko,  V. A. Recovering Sturm-Liouville operators on trees from spectra, Schriftenreiche des Fachbereichs Mathematik, SM-DU-684,
Universit\"at Duisburg-Essen (2009), 1--8.

\end{thebibliography}
\end{document}